\documentclass[10pt,twoside]{amsart}

\usepackage{amssymb}
\usepackage{amscd}
\usepackage{amsmath}
\usepackage[all]{xy}
\usepackage{mathrsfs}
\usepackage{graphicx}

\newtheorem{thm}{Theorem}[section]
\newtheorem{prop}[thm]{Proposition}

\newtheorem{lemma}[thm]{Lemma}

\newtheorem{defn}[thm]{Definition}
\newtheorem{example}[thm]{Example}
\newtheorem{examples}[thm]{Examples}
\newtheorem{lemma-def}[thm]{Lemma-Definition}
\newtheorem{theorem}{Theorem}[section]
\newtheorem{corollary}[thm]{Corollary}
\newtheorem{definition}[thm]{Definition}
\newtheorem{remark}[thm]{Remark}
\newtheorem{remarks}[thm]{Remarks}

\font\russ=wncyr10  1
\def\sha{\hbox{\russ\char88}}

\newcommand{\QQ}{\mathbb Q}

\newcommand{\ZZ}{\mathbb Z}
\newcommand{\RR}{\mathbb R}
\newcommand{\CC}{\mathbb C}

\newcommand{\bq}{\mathbb Q}

\newcommand{\bz}{\mathbb Z}

\newcommand{\A}{\mathfrak A}

\newcommand{\calF}{\mathcal F}

\newcommand{\calL}{\mathcal L}
\newcommand{\p}{\mathfrak{p}}

\DeclareMathOperator{\Sel}{Sel}

\usepackage[usenames]{color}

 \DeclareMathOperator{\Fit}{Fit}

 \DeclareMathOperator{\im}{im}

 \DeclareMathOperator{\nr}{nr}

\DeclareMathOperator{\Hom}{Hom} 
\DeclareMathOperator{\Ext}{Ext} 
 
\DeclareMathOperator{\cok}{cok} 
\DeclareMathOperator{\Tr}{Tr}

\renewcommand{\det}{\text{det}}

\def\YEAR{\year}\newcount\VOL\VOL=\YEAR\advance\VOL by-1995
\def\firstpage{1}\def\lastpage{1000}
\def\received{}\def\revised{}
\def\communicated{}

\makeatletter
\def\magnification{\afterassignment\m@g\count@}
\def\m@g{\mag=\count@\hsize6.5truein\vsize8.9truein\dimen\footins8truein}
\makeatother

\font\eightrm=cmr8
\font\caps=cmcsc10
\font\Caps=cmcsc10 scaled \magstep1   

\makeatletter
\@specialpagefalse\headheight=8.5pt
\def\DocMath{}
\renewcommand{\@evenhead}{%
    \ifnum\thepage>\lastpage\rlap{\thepage}\hfill%
    \else\rlap{\thepage}\slshape\leftmark\hfill{\caps\SAuthor}\hfill\fi}%
\renewcommand{\@oddhead}{%
    \ifnum\thepage=\firstpage{\DocMath\hfill\llap{\thepage}}%
    \else{\slshape\rightmark}\hfill{\caps\STitle}\hfill\llap{\thepage}\fi}%
\makeatother

\def\TSkip{\bigskip}
\newbox\TheTitle{\obeylines\gdef\GetTitle #1
\ShortTitle  #2
\SubTitle    #3
\Author      #4
\ShortAuthor #5
\EndTitle
{\setbox\TheTitle=\vbox{\baselineskip=20pt\let\par=\cr\obeylines%
\halign{\centerline{\Caps##}\cr\noalign{\medskip}\cr#1\cr}}%
        \copy\TheTitle\TSkip\TSkip%
\def\next{#2}\ifx\next\empty\gdef\STitle{#1}\else\gdef\STitle{#2}\fi%
\def\next{#3}\ifx\next\empty%
    \else\setbox\TheTitle=\vbox{\baselineskip=20pt\let\par=\cr\obeylines%
    \halign{\centerline{\caps##} #3\cr}}\copy\TheTitle\TSkip\TSkip\fi%
\centerline{\caps #4}\TSkip\TSkip%
\def\next{#5}\ifx\next\empty\gdef\SAuthor{#4}\else\gdef\SAuthor{#5}\fi%
\ifx\received\empty\relax
    \else\centerline{\eightrm Received: \received}\fi%
\ifx\revised\empty\TSkip%
    \else\centerline{\eightrm Revised: \revised}\TSkip\fi%
\ifx\communicated\empty\relax
    \else\centerline{\eightrm Communicated by \communicated}\fi\TSkip\TSkip%
\catcode'015=5}}\def\Title{\obeylines\GetTitle}
\def\Abstract{\begingroup\narrower
    \parskip=\medskipamount\parindent=0pt{\caps Abstract. }}
\def\EndAbstract{\par\endgroup\TSkip}

\long\def\MSC#1\EndMSC{\def\arg{#1}\ifx\arg\empty\relax\else
     {\par\narrower\noindent%
     1991 Mathematics Subject Classification: #1\par}\fi}

\long\def\KEY#1\EndKEY{\def\arg{#1}\ifx\arg\empty\relax\else
        {\par\narrower\noindent Keywords and Phrases: #1\par}\fi\TSkip}

\newbox\TheAdd\def\Addresses{\vfill\copy\TheAdd\vfill
    \ifodd\number\lastpage\vfill\eject\phantom{.}\vfill\eject\fi}
{\obeylines\gdef\GetAddress #1
\Address #2
\Address #3
\Address #4
\EndAddress
{\def\xs{5.0truecm}\parindent=0pt
\setbox0=\vtop{{\obeylines\hsize=\xs#1\par}}\def\next{#2}
\ifx\next\empty 
     \setbox\TheAdd=\hbox to\hsize{\hfill\copy0\hfill}
\else\setbox1=\vtop{{\obeylines\hsize=\xs#2\par}}\def\next{#3}
\ifx\next\empty 
     \setbox\TheAdd=\hbox to\hsize{\hfill\copy0\hfill\copy1\hfill}
\else\setbox2=\vtop{{\obeylines\hsize=\xs#3\par}}\def\next{#4}
\ifx\next\empty\ 
     \setbox\TheAdd=\vtop{\hbox to\hsize{\hfill\copy0\hfill\copy1\hfill}
                \vskip20pt\hbox to\hsize{\hfill\copy2\hfill}}
\else\setbox3=\vtop{{\obeylines\hsize=\xs#4\par}}
     \setbox\TheAdd=\vtop{\hbox to\hsize{\hfill\copy0\hfill\copy1\hfill}
                \vskip20pt\hbox to\hsize{\hfill\copy2\hfill\copy3\hfill}}
\fi\fi\fi\catcode'015=5}}\gdef\Address{\obeylines\GetAddress}

\begin{document}
\Title

On non-abelian higher special elements of $p$-adic representations

\ShortTitle
Non-abelian higher special elements
\SubTitle
\Author Daniel Macias Castillo  and Kwok-Wing Tsoi

\ShortAuthor Daniel Macias Castillo  and Kwok-Wing Tsoi
\EndTitle

\Abstract We develop a theory of `non-abelian higher special elements' in the non-commutative exterior powers of the Galois cohomology of $p$-adic representations. We explore their relation to the theory of organising matrices and thus to the Galois module structure of Selmer modules. In concrete applications, we relate our general theory to the formulation of refined conjectures of Birch and Swinnerton-Dyer type and to the Galois structure of Tate-Shafarevich and Selmer groups of abelian varieties.
\EndAbstract

\Address

Departamento de Matem\'aticas,
Universidad Aut\'onoma de Madrid and Instituto de Ciencias Matem\'aticas,
28049 Madrid, Spain.
daniel.macias@uam.es
daniel.macias@icmat.es

\Address

National Taiwan University,
Taipei 10617,
Taiwan (ROC).
kwokwingtsoi@ntu.edu.tw

\Address
\Address
\EndAddress
\section{Introduction}

\subsection{The general theory}

Before discussing the main objects of study in this article we must recall the theory of organising matrices.

\subsubsection{Organising matrices}

In the study of the classical (commutative) Iwasawa theory of elliptic curves, Mazur and Rubin first suggested in \cite{mr-1,mr0} the possibility of a theory of `organising modules' as a means of encoding detailed arithmetic information in a single matrix. A little later, Mazur and Rubin \cite{mr} succesfully associated (under certain hypotheses) such matrices to the corresponding Selmer complexes that were introduced by Nekov\'a\v r in \cite{nek}.

Subsequently, Burns and the first author \cite{omac} have both refined, and extended, this theory, in a way which associates a canonical family of `organising matrices' to general $p$-adic representations, considered with general (non-abelian) coefficients.

To be a little more precise, let $F/k$ be a finite Galois extension of global fields with Galois group $G$ and let $p$ be a prime number. Then to any $p$-adic Galois representation $T$ defined over $k$ one may associate certain \'etale cohomology complexes of $\ZZ_p[G]$-modules. These complexes are often `admissible' in the sense of loc. cit. (see also \S \ref{wa section} below) and in any such cases can be assigned a canonical family of organising matrices, with entries in $\ZZ_p[G]$, which encode a wide range of detailed information concerning the arithmetic of $T$ over $F/k$.

In this article we will give a generalisation of the construction of organising matrices. In addition, we will clarify their relationship to the theory of non-commutative higher Fitting invariants that has recently been developed by Burns and Sano \cite{nagm}. In this way we obtain a refinement of one of the main algebraic results (Corollary 3.3) of \cite{omac}. See \S \ref{omacsection} below for more details.

\subsubsection{Non-abelian higher special elements}

In the commutative setting Burns, Sano and the second author \cite{hse} have recently developed a theory of `higher special elements' as a generalisation of the notion of higher rank Euler systems. Such elements are also associated to admissible complexes in the sense of \cite{omac} and live in the higher exterior powers of the cohomology modules of the complexes. We recall that in arithmetic applications these modules are thus strictly related to the Galois cohomology of $p$-adic representations. 

Burns and Sano \cite{nagm} have also recently defined a natural notion of non-commutative higher rank Euler systems for $p$-adic representations $T$, relative to arbitrary Galois extensions of a number field over which $T$ is defined. They have then proved, under mild hypotheses, the existence of such non-commutative Euler systems whose rank depends explicitly on $T$.

The importance of this development is supported by, among other factors, a growing interest in the study of leading term conjectures that are relevant to non-abelian Galois extensions. These include the equivariant Tamagawa number conjecture of Burns and Flach \cite{bufl99}, the non-commutative Tamagawa number conjecture of Fukaya and Kato \cite{fukaya-kato}, the main conjecture of non-commutative Iwasawa theory for elliptic curves without complex multiplication of Coates, Fukaya, Kato, Sujatha and Venjakob \cite{cfksv}, the non-abelian Brumer-Stark conjectures formulated independently by Burns \cite{bdals} and by Nickel \cite{brumer}, or the recent formulation of a refined conjecture of Birch and Swinnerton-Dyer type by Burns and the first author \cite{rbsd}.

In order to arrive at the relevant notion of non-commutative higher rank Euler systems, Burns and Sano \cite{nagm} have also recently developed an algebraic theory of non-commutative exterior powers. The main objective of this article is to use this theory in order to define a completely general notion of `non-abelian higher special element'.

The theory of non-commutative exterior powers comes naturally equipped with canonical duality pairings and our definition gives a direct relationship between the images of non-abelian higher special elements under these pairings and the reduced norms of the corresponding organising matrices. 

By exploiting our understanding of the properties of organising matrices we are therefore able to prove that non-abelian higher special elements satisfy strong integrality properties and also encode delicate information regarding the Galois structure of Selmer modules of $p$-adic representations. These are the contents of our main algebraic result, given below as Theorem \ref{integrality}.

In this way, we hope to contribute to the future study of non-commutative higher rank Euler systems and of general leading term conjectures.

We emphasise that, exactly as in the commutative case considered in \cite{hse}, our construction of higher special elements does not depend on fixing `separable' tuples of elements (in the highest degree cohomology modules of admissible complexes) but rather arises from arbitrary choices of tuples. This fact makes potential arithmetic applications of our theory significantly finer than any specialisations currently present in the literature.


The degree of generality of our algebraic methods allows for subsequent applications in a wide range of natural arithmetic settings, including to the compact support cohomology complexes, the finite support cohomology complexes in the sense of Bloch and Kato \cite{bk} and the Nekov\'a\v r-Selmer complexes, that arise from very general $p$-adic representations. 

In this article we will however only focus on applications to certain classes of Selmer complexes associated to abelian varieties (see \S \ref{arithmetic} below). Before discussing these applications let us however mention other settings in which our algebraic results will play a significant role in future work.


\subsection{Arithmetic applications}

\subsubsection{Refined Stark conjectures}\label{introStick}

In work in progress of Burns, Seo and the first author \cite{bms}, our main algebraic result will be applied to the study of refined Stark conjectures and of the annihilation of ideal class groups by higher derivatives at $z=0$ of Artin $L$-functions.

A little more specifically, our main result will motivate the definition of canonical lattices of `non-abelian higher Stickelberger elements', and lead to the prediction that all such elements should provide integral annihilators of suitable ideal class groups. These predictions will be shown to hold unconditionally in the setting of global function fields.

Let us also note that such predictions will significantly extend and refine the non-abelian Brumer-Stark conjectures due to Burns \cite{bdals} and to Nickel \cite{brumer}.

In addition, in future work, we will also apply our algebraic results to study the annihilation of both suitable higher $K$-groups, and of wild kernels in higher $K$-theory, by higher derivatives at integer values of Artin $L$-functions. We aim to extend the constructions of conjectural annihilators of such modules that are carried out by Nickel in \cite{nickel2,nickelnew}.

\subsubsection{Refined conjectures of Birch and Swinnerton-Dyer type}\label{introRBSD}

Burns and the first author \cite{rbsd} have formulated a completely general `refined Birch and Swinnerton-Dyer conjecture' (or `refined BSD conjecture' in the sequel) for the Hasse-Weil-Artin $L$-series associated to an abelian variety $A$ defined over a number field $k$ and to a finite Galois extension $F$ of $k$.

This conjecture is equivalent to the relevant case of the equivariant Tamagawa number conjecture and is thus also compatible with the main conjecture of non-commutative Iwasawa theory formulated by Coates et al. in \cite{cfksv}.

Its formulation relies on a construction of canonical Nekov\'a\v r-Selmer complexes associated to choices of semi-local points on $A$. Under certain hypotheses, it may also be reformulated in terms of `classical Selmer complexes' that are closely related to the finite support cohomology that was introduced by Bloch and Kato in \cite{bk}.

In \cite[\S 8]{rbsd}, assuming that $F/k$ is abelian, Burns and the first author then study the congruence relations between the values at $z=1$ of higher derivatives of Hasse-Weil-Artin $L$-series, as well as their relation to the Galois structure of the Tate-Shafarevich and Selmer groups of $A$ over $F$, that are encoded in the refined BSD conjecture. (Both in loc. cit and in the sequel, it is always assumed that the relevant Tate-Shafarevich are finite.)

Our algebraic results will now allow us to extend this study to general finite Galois extensions $F/k$. In particular, in Theorem \ref{big conj} and the remarks that follow it, we will obtain an explicit description of the predictions that are encoded in the refined BSD conjecture in the general case. We are hopeful that it may be possible in the future to numerically test such predictions in non-abelian examples.


In \S \ref{dihedralsect} we then specifically consider (generalised) dihedral twists of elliptic curves over general number fields. 
By combining our approach with a result of Mazur and Rubin we are able to obtain strikingly explicit predictions for the derivatives of Hasse-Weil-Artin $L$-series of such twists, as we briefly discuss in the next section.

\subsubsection{Tate-Shafarevich groups of dihedral twists of elliptic curves}\label{introDih}

Let $A$ be an elliptic curve and let $F/k$ be generalised dihedral, in the sense of \cite{mr2}, of degree $2p^n$, say. Assume that $F/k$ is unramified at all places of $k$ at which $A$ has bad reduction. Let $K/k$ be the corresponding quadratic subextension of $F/k$, and assume that all $p$-adic places of $k$ split completely in $K/k$ and that the rank of $A(K)$ is odd.
In this setting, we first derive from the work of Mazur and Rubin in loc. cit. the existence of a point $Q$  in $A(F)$ on which ${\rm Gal}(F/k)$ acts in a specific manner.

Moreover, claim (ii) of Theorem \ref{dihedralthm} below (in combination with Theorem \ref{big conj}) gives a generalisation of the predictions studied by Burns, Wuthrich and the first author in \cite[Thm. 5.8]{bmw}. Namely we show that if the refined BSD conjecture is valid, then the derivatives of Hasse-Weil-Artin $L$-series, normalised by N\'eron-Tate heights associated to such a point $Q$, provide integral annihilators of the $p$-primary Tate-Shafarevich group $\sha(A_F)[p^\infty]$ of $A$ over $F$. We recall that, in addition to various other additional hypotheses, the group $\sha(A_F)[p^\infty]$ was assumed to vanish in \cite[Thm. 5.8]{bmw}, and $A(K)$ was assumed to have rank equal to $1$.

These predictions may be rendered strikingly explicit in many cases of interest and in various levels of generality.

\begin{examples}\label{introexamples}{\em \

\noindent{}(i)
By imposing some of the additional hypotheses of \cite[Thm. 5.8]{bmw}, but still without forcing the vanishing of $\sha(A_F)[p^\infty]$, we are led in claim (iii) of Theorem \ref{dihedralthm} below to predict that the elements $\mathcal{Q}_\psi$ that occur in \cite[Thm. 5.8]{bmw} (but constructed using our chosen point $Q$), as $\psi$ ranges over the irreducible complex character of ${\rm Gal}(F/k)$, combine to provide integral annihilators of $\sha(A_F)[p^\infty]$.


\noindent{}(ii)
As a concrete example of the discussion in the above paragraph, in Example \ref{heegner} below we explain how further specialisation leads to conjectural annihilation predictions for the elements $\mathcal{Q}_\psi$ corresponding to `higher Heegner points' $Q$ in (generalised) dihedral extensions $F$ of $\QQ$ that contain an imaginary quadratic field.

\noindent{}(iii) In a different direction, Christian Wuthrich kindly supplied us with the following concrete applications of Theorem \ref{dihedralthm}.
Set $k = \QQ$ and $K = \QQ(\sqrt{229})$ and write $F$ for the Galois closure of the field $L = \QQ(\alpha)$ with $\alpha^3-4\alpha+1 = 0$. Then $K \subset F$ and the group $G := {\rm Gal}(F/\QQ)$ is dihedral of order six. Let $A$ denote either of the curves 3928b1 (with equation $y^2 = x^3-x^2 + x + 4$) or 5864a1 (with equation $y^2 = x^3-x^2 -24x + 28$).

Then ${\rm rk}(A_\QQ)= 2$, ${\rm rk}(A_K) = {\rm rk}(A_L) = 3$ and ${\rm rk}(A_F) = 5$ while $\sha(A_K)[3^\infty]$ vanishes. These facts combine with \cite[Cor. 2.10(i)]{bmcw} to imply the $\ZZ_{3}[G]$-module $\ZZ_3\otimes_\ZZ A(F)$ is isomorphic to $$\ZZ_{3}[G](1-\delta) \oplus \ZZ_{3}\oplus \ZZ_{3},$$ with $\delta$ the unique non-trivial element in ${\rm Gal}(F/L)$. We let $Q$ be the image of $(1-\delta,0,0)$ under any such isomorphism. We also fix a non-trivial element $\gamma$ of ${\rm Gal}(F/K)$ and then define elements $T:=1+\gamma+\gamma^2$ and $U:=2-\gamma-\gamma^2$ of $\ZZ[G]$.

%
%

Then Theorem \ref{dihedralthm}(ii) and Theorem \ref{big conj} combine with the computation of `logarithmic resolvents' carried out in \cite[Prop. 8.11]{rbsd} to imply that,
if the refined BSD conjecture is valid for $A$ and $F/\QQ$, then the product
$$\beta\cdot(\Omega_A^-)^{-1}\cdot\tau^*(F/\QQ)\cdot\left(\frac{L'_S(A,\epsilon,1)}{2\langle T(Q),T(Q)\rangle_{A_F}}\cdot(1-\delta)T+3\frac{L'_S(A,\phi,1)}{\langle U(Q),U(Q)\rangle_{A_F}}\cdot U\right)$$
belongs to $\ZZ_3[G]$ and annihilates $\sha(A_{F})[3^\infty]$.

%
%

Here $\beta$ is any choice of element of the `ideal of denominators' associated to the order $\ZZ_3[G]$ by Burns and Sano \cite[Def. 3.4]{nagm}. Note also that the computations carried out by Johnston and Nickel \cite[\S 6.4]{jn} easily lead to explicit choices of $\beta$.

In addition $\Omega_A^-$ is the complex period of $A$, $\tau^*(F/\QQ)$ is the global Galois-Gauss sum of $F/\QQ$ defined in \cite[\S 4.2.1]{rbsd}, $\epsilon$ is the sign character of $G$ and $\phi$ is the irreducible character of $G$ of dimension $2$, $$S=\{2,3,229,\ell\}$$ with $\ell=491$ if $A$ is 3928b1 or with $\ell=733$ if $A$ is 5864a1 and both values at $1$ of derivatives are of the corresponding $L$-functions truncated by removing the Euler factors of primes in $S$. Further $\langle\,,\,\rangle_{A_F}$ denotes the N\'eron-Tate height pairing of $A$ defined relative to $F$.
}\end{examples}

\subsection{Structure of the article}

In \S \ref{prelims} we shall first recall the definition of the category of admissible complexes from \cite{omac}, and then also the construction of the specific admissible arithmetic complexes from \cite{rbsd} that will be relevant to our applications (as described in \S \ref{introRBSD}).

In \S \ref{chelomac} we then introduce the notion of a characteristic element, recall the theory of non-commutative Fitting invariants of Burns and Sano \cite{nagm}, and describe important links between both notions through a generalised construction of organising matrices that we give in \S \ref{omacsection}.

In \S \ref{nahse} we define our notion of non-abelian higher special elements after briefly recalling the theory of non-commutative exterior powers of Burns and Sano \cite{nagm}. As a preliminary step for the study of their finer integrality properties, we prove in \S \ref{ratsection} that these elements are rational in a natural sense.

In \S \ref{mainsection} we first state our main algebraic result together with two additional corollaries. We then proceed to prove our main result in the rest of the section.

In \S \ref{ts} we finally discuss the applications of our general theory to the arithmetic of abelian varieties, as outlined in \S \ref{introRBSD} and \ref{introDih}. 

\subsection{General notation}

For any ring $R$ we write $Z(R)$ for its centre.
Unless otherwise specified we regard all $R$-modules as left $R$-modules. 

We write $D(R)$ for the derived category of complexes of $R$-modules. If $R$ is noetherian, then we write $D^{\rm p}(R)$ for the full triangulated subcategory of $D(R)$ comprising complexes that are `perfect' (that is, isomorphic in $D(R)$ to a bounded complex of finitely generated projective $R$-modules).

For an abelian group $M$ we write $M_{\rm tor}$ for its torsion submodule and
set $M_{\rm tf}: = M/M_{\rm tor}$, which we regard as embedded in the associated space $\bq\otimes_{\bz}M$.

For a prime $p$ and natural number $n$ we write $M[p^n]$ for the subgroup $\{m \in M: p^nm =0\}$ of the Sylow $p$-subgroup $M[p^{\infty}]$ of $M_{\rm tor}$.
We set $M_p := \ZZ_p\otimes_\ZZ M$ and write $M_p^\wedge$ for the pro-$p$ completion of $M$.

If $G$ is a finite group we write ${\rm Ir}(G)$ for the set of irreducible complex characters of $G$. For a $G$-module $M$ we write $M^\vee$ for the Pontryagin dual of $M$, endowed with the natural contragredient action of $G$.







\subsection{Funding}

The first author acknowledges financial support from grants CEX2019-000904-S and PID2019-108936GB-C21 funded by MCIN/AEI/ 10.13039/501100011033.


\subsection{Acknowledgements} The authors are very grateful to David Burns for many interesting discussions, his generous encouragement, and some helpful comments on an earlier version of the article. 

The first author is also very grateful to Cornelius Greither, Henri Johnston, Andreas Nickel, Takamichi Sano, Stefano Vigni and Christian Wuthrich for many helpful conversations and correspondence. We particularly thank Andreas Nickel for pointing out a mistake in an earlier version of the proof of Lemma \ref{rationality}.

\section{Preliminaries}\label{prelims}



In this section we recall the category of admissible complexes and discuss specific aritmetic examples of such complexes that will play a key role in the sequel.

\subsection{Admissible complexes}\label{the complexes} We first introduce the categories of complexes to
which our main algebraic results apply. To do this we fix a
Dedekind domain $R$ of characteristic $0$ with field of fractions
$\calF$, a finite group $G$ and an $R$-order $\mathfrak{A}$ that spans a direct factor $A:=\calF\otimes_R\mathfrak{A}$ of the group ring $\calF[G]$.

\subsubsection{The category of admissible complexes}\label{wa section} We write $D^{\rm a}(\mathfrak{A})$ for the full subcategory of $D^{\rm p}(\mathfrak{A})$ comprising complexes $C = (C^i)_{i \in \bz}$ which satisfy
%
the following four assumptions:
\begin{itemize}
\item[(ad$_1$)] $C$ is an object of $D^{{\rm p}}(\mathfrak{A})$;
\item[(ad$_2$)] the Euler characteristic of $A\otimes_{\mathfrak{A}}C$ in the Grothendieck group $K_0(A)$
vanishes;
\item[(ad$_3$)] $C$ is acyclic outside degrees one and two;
\item[(ad$_4$)] $H^1(C)$ is $R$-torsion-free.
\end{itemize}

In the sequel we will refer to an object of $D^{\rm a}(\mathfrak{A})$ as an `admissible complex of $\mathfrak{A}$-modules'.



\begin{remark}{\em  In the case that $G$ is abelian, the category $D^{\rm a}(\A)$ plays a key role in the theory of higher special elements developed by Burns, Sano and the second author in \cite{hse}. However, we caution the reader that there is a slight difference of terminology in that objects of the category $D^{\rm a}(\A)$ defined above are in loc. cit. referred to as `strictly admissible' complexes. In the general case, the category $D^{\rm a}(\A)$ provided the setting for the theory of organising matrices developed by Burns and the first author in \cite{omac}.}

\end{remark}

\begin{example}\label{dual exam}{\em Write $\iota_\#$ for the $R$-linear anti-involution on $R[G]$ that satisfies $\iota_\#(g) = g^{-1}$ for all $g$ in $G$. Then for any idempotent $e$ of $Z(R[G])$ which is fixed by $\iota_\#$ the algebra $\mathfrak{A} := R[G]e$ is Gorenstein (in the sense of \cite[\S 2.1.2]{omac}, and with respect to the anti-involution that is obtained by restricting $\iota_\#$). In particular, the universal coefficient spectral sequence implies that the functor $C \mapsto C^*[-3]$  preserves the category $D^{\rm a}(\mathfrak{A})$.}
\end{example}

\subsubsection{Annihilation idempotents}\label{idempotents} If $C$ is an object of $D^{{\rm a}}(\mathfrak{A})$, then 
we write $e_0 = e_0(C)$ for the sum over all primitive idempotents
of $Z(A)$ that annihilate the
 module $H^1(A\otimes_{\mathfrak{A}}C)$. We note that the conditions (ad$_2$) and (ad$_3$) combine to imply that $H^1(A\otimes_{\mathfrak{A}}C)\cong H^2(A\otimes_{\mathfrak{A}}C)$. We use identical notation for complexes that satisfy (ad$_1$)-(ad$_3$) but do not satisfy the condition (ad$_4$).

\subsection{Arithmetic examples}\label{arithmetic}

For brevity we only discuss the specific arithmetic examples of admissible complexes that will be relevant to the applications of the general theory that will be given in this article. We refer the reader to both \cite[\S 2.2]{omac} and \cite[\S 2.2]{hse} for general discussions of how ubiquitous admissible complexes are in arithmetic.

In particular, both the Weil-\'etale cohomology complexes of the multiplicative group, and the \'etale cohomology complexes of cyclotomic representations, that are defined and studied by Burns, Kurihara and Sano in \cite{gm} and in \cite{iwa} respectively, constitute a source of natural and arithmetically significant admissible complexes. Considering our algebraic theory in these instances will lead to important arithmetic applications that will be developed in detail in future work, as mentioned in the introduction. 

\subsubsection{Nekov\'a\v r-Selmer complexes for abelian varieties}\label{selmerexample}
In this section we describe a construction of admissible complexes due to Burns and the first author \cite{rbsd}.

Let $F/k$ be a finite Galois extension of number fields with Galois group $G$.
Let $A$ be an abelian variety defined over $k$. We write $A^t$ for the dual abelian variety.

In the following result we write $X_\ZZ(A_F)$ for the integral Selmer group of $A$ over $F$ defined by Mazur and Tate in \cite{mt}.

We recall that, if the Tate-Shafarevich group $\sha(A_F)$ of $A$ over $F$ is finite, then $X_\ZZ(A_F)$ is a finitely generated $G$-module and there exists an isomorphism of $\hat \ZZ[G]$-modules $$\hat\ZZ\otimes_\ZZ X_\ZZ(A_F) \cong {\rm Sel}(A_F)^\vee$$ that is unique up to automorphisms of $X_\ZZ(A_F)$ that induce the identity map on both the submodule $X_\ZZ(A_F)_{\rm tor} = \sha(A_F)^\vee$ and quotient module $X_\ZZ(A_F)_{\rm tf} = \Hom_\ZZ(A(F), \ZZ)$. (Here $\hat\ZZ$ denotes the profinite completion of $\ZZ$).

In the sequel we write $S_\infty(k)$ for the set of archimedean places of $k$ and, for a given prime number $p$, also $S_p(k)$ for the set of places of $k$ which are $p$-adic. We write $S_{\rm bad}(A)$ for the set of places of $k$ at which $A$ has bad reduction and $S_{\rm ram}(F/k)$ for the set of places of $k$ which ramify in $F/k$. For a fixed set of places $S$ of $k$ we denote by $S(F)$ the set of places of $F$ which lie above a place in $S$.

In the following result we identify the category of finite $\ZZ_2[G]$-modules as an abelian subcategory of the category of $\ZZ[G]$-modules in the obvious way and write ${\rm Mod}^*(\ZZ[G])$ for the associated quotient category.

\begin{prop}\label{prop:perfect2} Assume that $\sha(A_F)$ is finite. Fix a `perfect Selmer structure' $\mathcal{X}$ for $A$ and $F/k$ in the sense of \cite[Def. 2.10]{rbsd}, and any finite set $S$ of places of $k$ with $$S_\infty(k)\cup S_{\rm ram}(F/k) \cup S_{\rm bad}(A)\subseteq S.$$

Then there exists a `Nekov\'a\v r-Selmer complex' ${\rm SC}_{S}(A_{F/k};\mathcal{X})$ in $D^{\rm p}(\ZZ[G])$, unique up to isomorphisms in $D^{\rm p}(\ZZ[G])$ that induce the identity map in all degrees of cohomology, that has all of the following properties.
\begin{itemize}
\item[(i)] ${\rm SC}_{S}(A_{F/k};\mathcal{X})$ is acyclic outside degrees $1, 2$ and $3$, and there is a canonical identification $H^3({\rm SC}_{S}(A_{F/k};\mathcal{X}))=(A(F)_{\rm tor})^\vee$.
\item[(ii)] In ${\rm Mod}^*(\ZZ[G])$ there exists a canonical injective homomorphism
\[ H^1({\rm SC}_{S}(A_{F/k};\mathcal{X})) \to A^t(F)\]
that has finite cokernel and a canonical surjective homomorphism
 \[ H^2({\rm SC}_{S}(A_{F/k};\mathcal{X}))\to X_\ZZ(A_F)\]
that has finite kernel.
\item[(iii)] For a given odd prime $\ell$ the object
\[ {\rm SC}_{S}(A_{F/k};\mathcal{X}(\ell)):=\ZZ_\ell\otimes_\ZZ {\rm SC}_{S}(A_{F/k};\mathcal{X}) \]
of $D^{\rm p}(\ZZ_\ell[G])$ lies in $D^{\rm a}(\ZZ_\ell[G])$ if and only is $A(F)$ contains no point of order $\ell$.
\end{itemize}
\end{prop}

\begin{remarks}{\em \
\noindent{}(i) Even for an odd prime $\ell$ for which $A(F)$ contains a point of order $\ell$, the general result \cite[Prop. 2.8]{hse} leads to a natural modification of ${\rm SC}_{S}(A_{F/k};\mathcal{X}(\ell))$ which belongs to the category $D^{\rm a}(\ZZ_\ell[G])$. See also \cite[Lem. 8.13(ii)]{rbsd} for more details.
\noindent{}(ii) The cohomology group $H^1({\rm SC}_{S}(A_{F/k};\mathcal{X}(\ell)))$ is in general a Selmer group in the sense of Mazur and Rubin \cite{mrkoly}. Moreover \cite[Prop. 2.12 (v) and Rem. 2.14]{rbsd} give fully explicit descriptions of the groups $H^1({\rm SC}_{S}(A_{F/k};\mathcal{X}))$ and $H^2({\rm SC}_{S}(A_{F/k};\mathcal{X}))$ in important cases.

\noindent{}(iii) Any choice of set $S$ as above and of an ordered $\QQ[G]$-basis $\omega_\bullet$ of the space of invariant differentials
$$H^0(A^t_F,\Omega^1_{A_F^t})\cong F\otimes_k H^0(A^t,\Omega^1_{A^t})$$ defines a perfect Selmer structure $\mathcal{X}_S(\omega_\bullet)$. See \cite{rbsd} for the details of this construction. The Nekov\'a\v r-Selmer complexes obtained from such perfect Selmer structures play a key role in the formulation of the central conjecture of loc. cit..

}\end{remarks}

\subsubsection{The classical Selmer complex of an abelian variety}\label{bkexample}

In this section we give an aditional construction of admissible complexes due to Burns and the first author \cite{rbsd} which, under suitable hypotheses, provides a more explicit alternative to the general Nekov\'a\v r-Selmer complexes discussed in the previous section. These complexes will play an important role, when combined with a result of Mazur and Rubin \cite[Thm. B]{mr2}, in Theorem \ref{dihedralthm}(iii) below.

We adopt the notation and setting of the previous section and also fix an odd prime number $p$. We will consider the following list of hypotheses on $A$, $F/k$ and $p$.

In this list we fix an intermediate field $K$ of $F/k$ with the property that $G_{F/K}$ is a Sylow $p$-subgroup of $G$.
In the sequel, for any non-archimedean place $v$ of a number field we write $\kappa_v$ for its residue field.

\begin{itemize}
\item[(H$_1$)] The Tamagawa number of $A_{/K}$ at each place in $S_{\rm bad}(A_{/K})$ is not divisible by $p$;
\item[(H$_2$)] $A_{/K}$ has good reduction at all $p$-adic places of $K$;
\item[(H$_3$)] For all $p$-adic places $v$ of $K$ that ramify in $F/k$, the reduction is ordinary and $A(\kappa_v)[p^\infty]= 0$;
\item[(H$_4$)] For all non-archimedean places $v$ of $K$ that ramify in $F/k$, we have $A(\kappa_v)[p^\infty]=0$;
\item[(H$_5$)] No place of bad reduction for $A_{/k}$ is ramified in $F/k$, i.e. $S_{\rm bad}(A_{/k}) \cap S_{\rm ram}(F/k) = \emptyset$.
\end{itemize}
We refer the reader to \cite[Rem. 6.1]{rbsd} for a further discussion of these hypotheses.

We write 
$Y_{F/k}$ for the module $\prod\ZZ$, with the product running over all $k$-embeddings $F\to k^c$, endowed with its natural action of $G\times G_k$.

We write $T_p(A^t)$ for the $p$-adic Tate module of  the dual variety $A^t$. Then the $p$-adic Tate module of the base change of $A^t$ through $F/k$ is equal to
\[ T_{p,F}(A^t) := Y_{F/k,p}\otimes_{\ZZ_p}T_p(A^t),\]
where $G$ acts on the first factor and $G_k$ acts diagonally.

For any subfield $E$ of $k$ and each non-archimedean place $v$ of $E$ we obtain a $G$-module by setting $$A^t(F_v):=\prod_{w\mid v}A^t(F_w),$$ with $w$ running over all places of $F$ which divide $v$.

\begin{definition}\label{bkdefinition}{\em We fix a finite set of non-archimedean places $\Sigma$ of $k$ with
\[ S_p(k)\cup S_{\rm ram}(F/k) \cup S_{\rm bad}(A)\subseteq \Sigma.\]
We define the `classical $p$-adic Selmer complex' ${\rm SC}_{\Sigma,p}(A_{F/k})$ to be the mapping fibre of the following morphism in $D(\ZZ_p[G])$:
\begin{equation*}\label{bkfibre}
R\Gamma(\mathcal{O}_{k,S_\infty(k)\cup\Sigma},T_{p,F}(A^t)) \oplus \bigl(\bigoplus_{v\in\Sigma}A^t(F_v)^\wedge_{p}\bigr)[-1]  \xrightarrow{(\lambda,\kappa)} \bigoplus_{v \in \Sigma} R\Gamma (k_v, T_{p,F}(A^t)).\end{equation*}
Here $\lambda$ is the natural diagonal localisation morphism in \'etale cohomology and $\kappa$ is induced by the sum of the local Kummer maps $A^t(F_v)_p^\wedge\to H^1(k_v,T_{p,F}(A^t))$ (and the fact that for each $v\in\Sigma$ the group $H^0(k_v, T_{p,F}(A^t))$ vanishes).}
\end{definition}

\begin{remark}\label{indeptofSigma}{\em By Lem. 2.5 (and Rem. 2.6) in \cite{rbsd}, the classical $p$-adic Selmer complex is independent, in a natural sense, of the choice of set of places $\Sigma$ fixed in Definition \ref{bkdefinition}. For this reason, in the sequel we shall usually simply denote it by ${\rm SC}_{p}(A_{F/k})$.}
\end{remark}

In the following result we write ${\rm Sel}_p(A_F)$ for the classical $p$-primary Selmer group of $A$ over $F$.

\begin{prop}\label{bkprop}Asume that $A$, $F/k$ and $p$ satisfy the hypotheses (H$_1$)-(H$_5$) and that $\sha(A_F)[p^\infty]$ is finite. 

Then the classical $p$-adic Selmer complex ${\rm SC}_p(A_{F/k})$ belongs to $D^{\rm p}(\ZZ_p[G])$ and is acyclic outside degrees one, two and three, with canonical identifications
\begin{equation*}\label{bkcohomology}H^i({\rm SC}_p(A_{F/k}))=\begin{cases}A^t(F)_p,\,\,\,\,\,\,\,\,\,\,\,\,\,\,\,\,\,\,i=1,\\ {\rm Sel}_p(A_F)^\vee,\,\,\,\,\,\,\,\,\,\,i=2,\\ A(F)[p^\infty]^\vee,\,\,\,\,\,\,\,\,i=3.\end{cases}\end{equation*}
In particular, ${\rm SC}_p(A_{F/k})$ belongs to $D^{\rm a}(\ZZ_p[G])$ if and only if neither $A(F)$ nor $A^t(F)$ contains a point of order $p$.
\end{prop}

\section{Characteristic elements, Fitting invariants and organising matrices}\label{chelomac}

\subsection{Characteristic elements}\label{dets} In this section we follow \cite{omac} in order to associate a natural notion of characteristic element to complexes in $D^{\rm a}(\mathfrak{A})$.

For any field $E$ that contains $\calF$ and any $\mathfrak{A}$-module $M$, resp. homomorphism of $\mathfrak{A}$-modules $\phi$, we write $M_E$ for the associated $E\otimes_{R}\mathfrak{A}$-module $E\otimes_{R}M$, resp. $\phi_E$ for the associated homomorphism $E\otimes_{R}\phi$ of $E\otimes_{R}\mathfrak{A}$-modules. We also use similar abbreviations for complexes, and morphisms of complexes, of $\mathfrak{A}$-modules.

\subsubsection{Relative $K$-theory}

For any field $E$ as above, we write $K_0(\mathfrak{A},\mathfrak{A}_E)$ for the relative algebraic $K_0$-group of the ring inclusion $\mathfrak{A}\subset \mathfrak{A}_E$. We recall that this group comprises elements of the form $[P_1,P_2,\theta]$ where $P_1$ and $P_2$ are finitely generated projective $\mathfrak{A}$-modules and $\theta$ is an isomorphism of $\mathfrak{A}_E$-modules from $P_{1,E}$ to $P_{2,E}$.

We further recall that there exists a canonical exact commutative diagram of abelian groups of the form

\begin{equation}\label{ktheory}\begin{CD} K_1(\mathfrak{A}) @> \partial^2_{\mathfrak{A},\calF} >> K_1(A) @> \partial^1_{\mathfrak{A},\calF} >> K_0(\mathfrak{A},A) @> \partial^0_{\mathfrak{A},\calF} >> K_0(\mathfrak{A}) @> \partial^{ -1}_{\mathfrak{A},\calF} >> K_0(A)\\
@\vert  @V \iota_E VV @V \iota_E' VV @\vert @V \iota_E'' VV\\
K_1(\mathfrak{A}) @> \partial^2_{\mathfrak{A},E} >> K_1(\mathfrak{A}_E) @> \partial^1_{\mathfrak{A},E} >> K_0(\mathfrak{A},\mathfrak{A}_E)@> \partial^0_{\mathfrak{A},E} >> K_0(\mathfrak{A}) @> \partial^{-1}_{\mathfrak{A},E} >> K_0(\mathfrak{A}_E).\end{CD}\end{equation}

Here the homomorphisms $\iota_E, \iota_E'$ and $\iota_{E}''$ are induced by the inclusion $A \subseteq \mathfrak{A}_E$ and are injective (indeed, we shall often regard these maps as inclusions). The homomorphisms $\partial^2_{\mathfrak{A},E}$ and $\partial^{-1}_{\mathfrak{A},E}$ are induced by the inclusion $\mathfrak{A} \subset \mathfrak{A}_E$,  $\partial^1_{\mathfrak{A},E}$ is the homomorphism that sends the class of an automorphism $\phi$ of $\mathfrak{A}_E^n$ to $[\mathfrak{A}^n,\mathfrak{A}^n,\phi]$ and $\partial^0_{\mathfrak{A},E}$ sends each element $[P_1,P_2,\theta]$ to $[P_1] - [P_2]$.

For any such field $E$ that is large enough (to contain, say, either $\QQ_p$ for some prime $p$, or to contain $\RR$), we write
\[ \delta_{\mathfrak{A},E}:Z(\mathfrak{A}_E)^\times\to K_0(\mathfrak{A},\mathfrak{A}_E)\]
for the `extended boundary homomorphism' that is defined in \cite{bufl99} and we recall that
\[\delta_{\mathfrak{A},E} \circ {\rm nr}_{\mathfrak{A}_E} =\partial^1_{\mathfrak{A},E}\]
where ${\rm nr}_{\mathfrak{A}_E}$ denotes the homomorphism $K_1(\mathfrak{A}_E) \to Z(\mathfrak{A}_E)^\times$ that is induced by taking reduced norms.

\subsubsection{Characteristic elements}\label{char}

In the sequel, for any ring $\Lambda$ and any (left) $\Lambda$-modules $M$ and $N$ we write ${\rm Is}_\Lambda(M,N)$ for the set of $\Lambda$-module isomorphisms $M \to N$.

%
%
For each cohomologically-bounded complex of $\mathfrak{A}$-modules $C$ and each field $E$ that contains $\calF$ we define the set of `$E$-trivialisations' of $C$ by setting
\[ \tau(C_E):={\rm Is}_{\mathfrak{A}_E}(\bigoplus_{i\in \bz}H^{2i}(C)_E,\bigoplus_{i\in \bz}H^{2i+1}(C)_E).\]

\begin{example}\label{trivialisationex}{\em
In the setting of both \S \ref{selmerexample} and \S\ref{bkexample}, the N\'eron-Tate height pairing of $A$, defined relative to the field $F$, defines an $\RR[G]$-isomorphism $$h_{A,F/k}:A^t(F)_{\RR}\to\Hom_\RR(A(F)_\RR,\RR).$$ If $\sha(A_F)$ is finite then, since $X_\ZZ(A/F)_{\rm tf}=\Hom_\ZZ(A(F),\ZZ)$, the map $h_{A,F/k}$ gives an $\RR$-trivialisition of the Nekov\'a\v r-Selmer complexes discussed in Proposition \ref{prop:perfect2}. Under the relevant hypotheses it also induces, for each isomorphism $\CC\cong\CC_p$, a $\CC_p$-trivialisation of the classical $p$-adic Selmer complex.
}\end{example}

We next recall that to each pair $(C,t)$ comprising a complex $C$ in $D^{\rm p}(\mathfrak{A})$ and an $E$-trivialisation $t$ of $C$ one can associate a canonical `refined' Euler characteristic $\chi^{\rm ref}_{\A,E}(C,t)$ in $K_0(\mathfrak{A},\mathfrak{A}_E)$. (For explicit details of this construction in the relevant special case see the argument given in \cite[\S 4.1]{omac}; for details in a more general context see, for example, \cite[\S2.8]{bufl99}.)

\begin{definition}\label{characteristicelement}{\em For $C$ in $D^{\rm p}(\mathfrak{A})$ and $t$ in $\tau(C_E)$ we define a `characteristic element for the pair $(C,t)$' to be any element $\mathcal{L}$ of $Z(\mathfrak{A}_E)^\times$ which satisfies
\begin{equation*}\label{char def} \delta_{\mathfrak{A},E}(\mathcal{L}) =  -\chi^{\rm ref}_{\A,E}(C,t).\end{equation*}

We also define a `characteristic element for $C$' to be a characteristic element for $(C,t)$ for any choice of field $E$ and of $E$-trivialisation $t$ of $C$.}
\end{definition}

From the lower exact sequence in (\ref{ktheory}) it is then clear that characteristic elements for $(C,t)$ are unique up to multiplication by elements of ${\rm nr}_{\mathfrak{A}_E}(\im(\partial^2_{\mathfrak{A},E}))$. In this regard we also recall that if $\mathfrak{A}$ is semi-local (which is automatic if $R$ is a discrete valuation ring), then the natural homomorphism $\mathfrak{A}^\times \to K_1(\mathfrak{A})$ is surjective.


We next record a property of characteristic elements that will be useful in the proof of our main algebraic result. We refer to \cite[\S 2.3.3]{omac} for the proof of this fact, which relies on Bass' Theorem.

\begin{lemma}\label{lam lemma} Let $R$ be a discrete valuation ring and $E$ be any field that contains $\calF$. Fix a complex $C$ in $D^{\rm a}(\mathfrak{A})$, write $e_0$ for the idempotent $e_0(C)$ defined in \S\ref{idempotents} and set $\mathfrak{A}_0 := \mathfrak{A}e_0$, $A_{E,0} := \mathfrak{A}_Ee_0$ and $C_0 := \mathfrak{A}_0\otimes_{\mathfrak{A}}^{\mathbb{L}}C$.

Then $C_0$ belongs to $D^{\rm a}(\mathfrak{A}_0)$ and for any characteristic element $\mathcal{L}_0$ in $Z(A_{E,0})^\times$ of $C_0$ there exists a characteristic element $\mathcal{L}$ of $C$ such that $e_0\mathcal{L} = \mathcal{L}_0$.\end{lemma}

\subsection{Non-commutative higher Fitting invariants}

In this section we briefly review the theory of non-commutative higher Fitting invariants introduced by Burns and Sano in \cite{nagm}. The following definition of `Whitehead orders' plays a key role in these constructions.

\subsubsection{The Whitehead order}

For any prime ideal $\p$ of $R$ we write $R_{(\p)}$ for the localisation of $R$ at $\p$ and for any $\A$-module $M$ we then set $M_{(\p)}:=R_{(\p)}\otimes_R M$.

\begin{definition}\label{centralorder}{\em For each prime ideal $\p$ of $R$ the `Whitehead order' $\xi(\A_{(\p)})$ is the $R_{(\p)}$-submodule of $Z(A)$ that is generated by the reduced norms ${\nr}_A(M)$ as $M$ runs over the set $\bigcup_{n\geq 0}M_n(\A_{(\p)})$ of all square matrices with coefficients in $\A_{(\p)}$.

The Whitehead order of $\A$ is then defined by the intersection
$$\xi(\A):=\bigcap_\p \xi(\A_{(\p)}),$$
where $\p$ runs over all prime ideals of $R$.}
\end{definition}

By \cite[Lem. 3.2]{nagm} we know that $\xi(\A)$ is indeed an $R$-order in $Z(A)$, which furthermore satisfies $\xi(\A)_{(\p)}=\xi(\A_{(\p)})$ for each prime ideal $\p$ of $R$.

\subsubsection{The higher Fitting invariants of a matrix}
Let $M$ be any matrix in $M_{d\times d'}(A)$ with $d\geq d'$. Then for any integer $t$ with $0\leq t\leq d'$ and any $\varphi=(\varphi_i)_{1\leq i\leq t}$ in $\Hom_\A(\A^d,\A)^t$ we write ${\rm Min}_\varphi^{d'}(M)$ for the set of all $d'\times d'$-minors of the matrices $M(J,\varphi)$ that are obtained from $M$ by choosing any $t$-tuple of integers $J=\{i_1,\ldots,i_t\}$ with $1\leq i_1\leq\ldots\leq i_t\leq d'$, and setting
\begin{equation*}\label{MJvarphi}M(J,\varphi)_{ij}:=\begin{cases}\varphi_a(b_i),\,\,\,\,\,\,\,\,\text{ if }j=i_a\text{ with }1\leq a\leq t,\\ M_{ij},\,\,\,\,\,\,\,\,\,\,\,\,\,\,\text{ otherwise, }\end{cases}\end{equation*}
where $\{b_i\}$ denotes the standard basis of $\A^d$.

For any non-negative integer $a$ the `$a$-th (non-commutative) Fitting invariant' of $M$ is defined to be the ideal of $\xi(\A)$ obtained by setting
$${\rm Fit}_\A^{a}(M):=\xi(\A)\cdot\{\nr_A(N):N\in{\rm Min}_\varphi^{d'}(M),\varphi\in\Hom_\A(\A^d,\A)^t,t\leq a\}.$$

\subsubsection{The higher Fitting invariants of a presentation}

A `free presentation' $\Pi$ of a finitely generated $\A$-module $Z$ is an exact sequence of $\A$-modules of the form
\begin{equation}\label{freepres}F^1\stackrel{\theta}{\to}F^2\to Z\to 0\end{equation} in which the $\A$-modules $F^1$ and $F^2$ are finitely generated and free and (without loss of generality) one has ${\rm rk}_\A F^1\geq{\rm rk}_\A F^2$. The free presentation $\Pi$ is said to be quadratic if ${\rm rk}_\A F^1={\rm rk}_\A F^2$.

Th $a$-th Fitting invariant $\Fit^a_\A(\Pi)$ of $\Pi$ is defined to be $\Fit^a_\A(M_\theta)$ for any matrix $M_\theta$ which represents $\theta$ with respect to any choice of $\A$-bases of $F^1$ and $F^2$.

We recall that a finitely generated $\A$-module $F$ is said to be locally-free if the localisation $F_{(\p)}$ is a free $\A_{(\p)}$-module (or equivalently if the completion $F_\p$ is a free $\A_\p$-module) for every prime ideal $\p$ of $R$.

A `locally-free presentation' $\Pi$ of $Z$ is an exact sequence of the form (\ref{freepres}) in which the $\A$-modules $F^1$ and $F^2$ are only assumed to be finitely generated and locally-free. The locally-free presentation $\Pi$ is then said to be `locally-quadratic' if ${\rm rk}_\A F^1={\rm rk}_\A F^2$ (we recall that the rank of a locally-free $\A$-module is indeed a well-defined invariant).

By localising the sequence $\Pi$ at a prime ideal $\p$ of $R$ one obtains a free resolution $\Pi_{(\p)}$ of the $\A_{(\p)}$-module $Z_{(\p)}$ and one then defines the $a$-th Fitting invariant of $\Pi$ to be
$$\Fit^a_\A(\Pi):=\bigcap_\p \Fit^a_{\A_{(\p)}}(\Pi_{(\p)}),$$ where the intersection runs over all prime ideals $\p$ of $R$ and takes place in $Z(A)$.

We note in passing that, for $a=0$, there is an explicit relationship between $\Fit^0_\A(\Pi)$ and the Fitting invariant defined by Nickel in \cite{nickel}, as explained in \cite[Prop. 3.13 (i)]{nagm}.



\subsubsection{The total higher Fitting invariants}

In order to associate to the module $Z$ a finer invariant that will be crucial to our approach, we say that a free presentation $$\Pi':F^{1,'}\stackrel{\theta'}{\to}F^{2,'}\to Z'\to 0$$ of a finitely generated $\A$-module $Z'$ is `finer' that the free presentation $\Pi$ of $Z$ given by (\ref{freepres}) if both ${\rm rk}_\A(F^1)={\rm rk}_\A(F^{1,'})$ and ${\rm rk}_\A(F^2)={\rm rk}_\A(F^{2,'})$ and there exists an isomorphism $F^{2,'}\cong F^2$ which induces a well-defined surjective homomorphism $Z'\to Z$.

We define the `total $a$-th Fitting invariant' of $\Pi$ to be
$$\Fit^{a,{\rm tot}}_\A(\Pi):=\sum_{\Pi'}\Fit_\A^a(\Pi')$$ where in the sum $\Pi'$ runs over all free presentations that are finer than $\Pi$.

For a locally-free presentation $\Pi$ we define the total $a$-th Fitting invariant of $\Pi$ to be $$\Fit^{a,{\rm tot}}_\A(\Pi):=\bigcap_\p \Fit^{a,{\rm tot}}_{\A_{(\p)}}(\Pi_{(\p)}).$$

We recall an useful property of this invariant. In this result, 
we denote by $\mathfrak{D}(\A)$ the `ideal of denominators' that is introduced by Burns and Sano in \cite[Def. 3.4]{nagm} (and denoted $\delta(\A)$ in loc. cit., which in this article could unfortunately clash with the standard notation for the extended boundary homomorphism).

This result is then a consequence of the stronger \cite[Thm. 3.17 (iii)]{nagm}.

\begin{lemma}\label{BSlemma}{{\em (Burns-Sano)}} Let $\Pi$ be a locally-free presentation of an $\A$-module $Z$. Then 
one has $$\mathfrak{D}(\A)\cdot\Fit^{0}_\A(\Pi)\subseteq\mathfrak{D}(\A)\cdot\Fit^{0,{\rm tot}}_\A(\Pi)\subseteq{\rm Ann}_\A(Z).$$
\end{lemma}

\subsection{A construction of organising matrices}\label{omacsection}

If $C$ is an object of $D^{{\rm a}}(\ZZ_p[G])$ then Burns and the first author associate in \cite[Thm. 3.1]{omac} a family of `weakly-organising matrices' to $C$; an additional specification of data then also leads via Theorem 3.11 in loc. cit. to a more restrictive family of `organising matrices' associated to $C$. Let us recall that the original motivation of this theory lies in the construction by Mazur and Rubin of `organising modules' in the (commutative) Iwasawa theory of elliptic curves (see \cite{mr}).

In this section we give a generalisation of the construction of such matrices that will play a key role in the proof of Theorem \ref{integrality} below. We only discuss the properties of our organising matrices that will be of subsequent use in this article (but see also Remark \ref{rgremark} below for a comparison to one of the main results of \cite{omac}).

\subsubsection{Statement of the main results}

Throughout this section we assume that $R$ is a discrete valuation ring. As a natural generalisation of admissible complexes, we assume to be given a complex $D$ of $\A$-modules of the form \begin{equation}\label{complexoftheformD}D^0\stackrel{\delta^0}{\to}D^1\stackrel{\delta^1}{\to} D^2\end{equation} in which the first term is placed in degree zero and each module is finitely generated and free of rank $a$, $d$ and $d-a$ respectively, always assuming also that $d\geq a$. We assume further that $D$ is acyclic outside degrees one and two.

We then fix data of the following form:
\begin{itemize}\item[(D$_1$)] an element $z$ of the group ${\rm Ann}_\A(\Ext^2_\A(H^2(D),\A))$;
\item[(D$_2$)] and an ordered $a$-tuple of homomorphisms $$\phi_\bullet=\{\phi_1,\ldots,\phi_a\}\subseteq\Hom_\A(\ker(\delta^1),\A).$$
\end{itemize}

\begin{theorem}\label{almost there} Assume that $R$ is a discrete valuation ring and fix a complex of $\A$-modules $D$ and elements $z$ and $\phi_\bullet$ as above. We also fix any $\A$-bases $\{b_i^j\}$ of $D^j$ for $j=0,1,2$.

Then there exists a matrix $\Phi=\Phi_{D,z,\phi_\bullet}\in M_d(\A)$ that satisfies the following conditions:
\begin{itemize}\item[(i)] $\Phi$ is of the form \[
\left(
\begin{array}{c|c}
 \star & \Delta^1
\end{array}
\right)
\] where $\Delta^1\in M_{d,d-a}(\A)$ is the matrix of $\delta^1$ with respect to the chosen bases of $D^1$ and $D^2$.
\item[(ii)] We set $\Lambda:=\left(\phi_j(\delta^0(b_i^0))\right)_{1\leq i,j\leq a}\in M_a(\A)$. Then for any characteristic element $\calL$ of $D$, the element
\begin{equation}\label{omaccharelement}{\rm nr}_A(z)^a\cdot{\rm nr}_A\bigl(\Lambda\bigr)\cdot e_0(D)\cdot\mathcal{L}\end{equation} belongs to ${\rm nr}_A(\A^\times)\cdot\nr_A(\Phi)$. (Here $e_0(D)$ is the idempotent defined in \S \ref{idempotents}).
\end{itemize}
\end{theorem}

Before proving Theorem \ref{almost there}, we state a direct consequence of this result. 

\begin{corollary}\label{omaccor} For any characteristic element $\calL$ of $D$, the element (\ref{omaccharelement}) belongs to $${\rm Fit}^{0,{\rm tot}}_\A(\Pi)\cap\Fit_\A^a\left( 
\left(
\begin{array}{c|c}
\begin{array}{c}I_a\\ 0\end{array} & \Delta^1
\end{array}
\right)
\right)$$
where $I_a$ is the identity $a\times a$ matrix, $0$ is the trivial $(d-a)\times a$ matrix and $\Pi$ denotes the canonical free presentation $$D^0\oplus D^1\stackrel{{\rm id}\oplus\delta^1}{\to}D^0\oplus D^2{\to}H^2(D)\to 0$$ of the $\A$-module $H^2(D)$.
\end{corollary}

\begin{proof} Since claim (i) of Theorem \ref{almost there} implies that $\Phi$ coincides with $$M:=\left(
\begin{array}{c|c}
\begin{array}{c}I_a\\ 0\end{array} & \Delta^1
\end{array}
\right)$$ in all but the first $a$ columns, one has \begin{equation}\label{fita}{\rm nr}_A(\A^\times)\cdot\nr_A(\Phi)\subseteq\xi(\A)\cdot\nr_A(\Phi)\subseteq\Fit^a_\A(M).\end{equation}

We next consider the following commutative diagram with exact rows:
\[\begin{CD}
 D^0\oplus D^1 @> 0+\Phi >> D^0 \oplus D^2 @> >> {\rm cok}(\Phi) @> >> 0\\
 @V\Phi^0 VV @V {\rm id}VV @V\epsilon VV \\
D^0\oplus D^1 @> {\rm id}\oplus\delta^1 >> D^0\oplus D^2 @> >> H^2(D
 ) @> >> 0.\end{CD}\]
Here $\Phi$ is interpreted as a homomorphism $D^1\to D^0\oplus D^2$ through the fixed bases; the map $\Phi^0$ is defined as the composition $$D^0\oplus D^1\stackrel{((0+\Phi),(0+{\rm id}))}{\longrightarrow}(D^0\oplus D^2)\oplus D^1\stackrel{\pi_0\oplus{\rm id}}{\to}D^0\oplus D^1$$
where $\pi_0$ denotes the projection to the $D^0$ component; the left-hand side square commutes by Theorem \ref{almost there}(i); and the surjective map $\epsilon$ is defined by the commutativity of the right-hand side square.

We write $\Pi'$ for the free presentation of the $\A$-module $\cok(\Phi)$ that is given by the top row of this diagram. In particular, the presentation $\Pi'$ is finer than the presentation $\Pi$. It follows that
\begin{equation}\label{fit0tot}{\rm nr}_A(\A^\times)\cdot\nr_A(\Phi)\subseteq\xi(\A)\cdot\nr_A(\Phi)=\Fit^0_\A(\Pi')\subseteq\Fit^{0,{\rm tot}}_\A(\Pi).\end{equation}

The inclusions (\ref{fita}) and (\ref{fit0tot}) combine with claim (ii) of Theorem \ref{almost there} to imply the validity of the claimed containment.
\end{proof}

\begin{remark}\label{rgremark}{\em Let $D$ be a complex of $R[G]$-modules which satisfies conditions (ad$_1$), (ad$_2$) and (ad$_3$). Then one may fix a representative of $D$ of the form (\ref{complexoftheformD}). Furthermore, the group $\Ext^2_{R[G]}(H^2(D),R[G])$ vanishes so one may take $z=1$ in (\ref{omaccharelement}).

If in addition $D$ satisfies the condition (ad$_4$) then one may fix a representative of $D$ of the form (\ref{complexoftheformD}) for which $D^0=0$, so that $a=0$ and the term ${\rm nr}_A(\Lambda)$ may be taken to be equal to 1 in (\ref{omaccharelement}). One thus sees that Corollary \ref{omaccor} provides a refinement and generalisation of \cite[Cor. 3.3]{omac}.

In fact, our ability to apply Corollary \ref{omaccor} over more general $R$-orders $\A$ will be crucial in the sequel, and will also be helpful in future applications.
}
\end{remark}

\subsubsection{The proof of Theorem \ref{almost there}} The differential $\delta^0$ is injective and, since the groups $e_0(H^2(D)_\calF)$ and $e_0(H^1(D)_\calF)$ both vanish, there exists a direct sum decomposition $e_0D^1_\calF=V_1^1\oplus V_2^1$ so that the maps $e_0\delta^0_\calF$ and $e_0\delta^1_\calF$ give isomorphisms $e_0D^0_\calF\cong V_1^1$ and $V_2^1\cong e_0D^2_\calF$ respectively. We can therefore fix an isomorphism of $A_E$-modules \begin{equation}\label{20}\iota:(D^0\oplus D^2)_E\to D^1_E\end{equation} whose restriction coincides with the scalar extension of the isomorphism $$e_0D^0_\calF\oplus e_0D^2_\calF\cong e_0D^1_\calF=V_1^1\oplus V_2^1$$ given by $(e_0\delta^0_\calF,(e_0\delta^1_\calF)^{-1})$.

We now assume to be given $t$ in ${\rm Is}_{A_E}(H^2(D)_E,H^1(D)_E)$ and a characteristic element $\calL$ for the pair $(D,t)$. Then the same argument used to prove \cite[Thm. 3.1(iv)]{omac} (under the simplifying assumption that $D$ is acyclic in degree 3) implies that there is $u_\calL\in{\rm nr}_A(\A^\times)$ with the property that \begin{equation}\label{21}e_0\calL={\rm nr}_{A_E}(\iota^{-1})u_\calL.\end{equation}

We next apply the functor $\Hom_\A(-,\A)$ to the tautological short exact sequences
\[\begin{cases}
 0 \to Z^1(D) \to D^1 \to B^2(D) \to 0\\
0\to B^2(D)\to D^2\to H^2(D)\to 0.
\end{cases}\]
In particular, since the groups $\Ext_\A^j(D^i,\A)$ vanish for each $j\geq 1$ and each $i\in\{1,2,3\}$, we obtain an exact sequence 
$$\Hom_\A(D^1,\A)\stackrel{\kappa}{\to}\Hom_\A(Z^1(D),\A){\to}\Ext_\A^1(B^2(D),\A)\to 0$$ and an isomorphism $$\Ext_\A^1(B^2(D),\A)\cong\Ext_\A^2(H^2(D),\A).$$
Hence for $1\leq j\leq a$ there exist homomorphisms $\varphi_j$ in $\Hom_\A(D^1,\A)$ with $$\varphi_j|_{Z^1(D)}=\kappa(\varphi_j)=z\cdot\phi_j.$$

We define $\phi$ to be the element of $\Hom_\A(D^1,D^0)$ that maps each element $w$ of $D^1$ to $$\sum_{i=1}^{i=a}\varphi_i(w)\cdot b^0_i$$ where $\{b_i^0\}_{1\leq i\leq a}$ is our fixed basis of $D^0$, and consider the homomorphism $D^1\to D^0\oplus D^2$ that is given by the direct sum $\phi\oplus\delta^1$.

Now, by explicitly comparing this map to the isomorphism $\iota$ defined in (\ref{20}) one computes that on $e_0D^0_E\oplus e_0D^2_E$ there is an equality of functions
\begin{equation}\label{composition}e_0(\phi\oplus\delta^1)_E\circ e_0(\iota)=(e_0(\phi\circ\delta^0)_E,{\rm id}_{e_0D^2_E})\end{equation} and for each basis element $b_i^0$ one has 
\begin{equation}\label{computation}(\phi\circ\delta^0)(b^0_i)=\sum_{j=1}^{j=a}\varphi_j(\delta^0(b^0_i))\cdot b^0_j=\sum_{j=1}^{j=a}z(\phi_j\circ\delta^0)(b^0_i)\cdot b^0_j.
\end{equation}
We set $\lambda_{ij}:=(\phi_j\circ\delta^0)(b^0_i)$.

For any fixed choice of bases $\{b^1_i\}$ and $\{b^2_i\}$ of $D^1$ and $D^2$ respectively we write $\Delta^1$ for the matrix of $\delta^1$ and $\Phi$ for the matrix of $\phi\oplus\delta^1$, as computed with respect to these choices and $\{b^0_i\}$. Claim (i) is then trivially satisfied.

We next note that $\cok(\phi\oplus\delta^1)$ surjects canonically onto $H^2(D)=\cok(\delta^1)$.
In particular, for any primitive central idempotent $e$ of $A$ one has that $e(\Phi)$ is invertible over $Ae$ only if $e=ee_0$, so we deduce that ${\rm nr}_A(\Phi)=e_0{\rm nr}_A(\Phi)$. Combining this equality with (\ref{21}), (\ref{composition}) and (\ref{computation}) we finally find that
\begin{align*}\label{24}{\rm nr}_A(z)^a{\rm nr}_A(\Lambda)e_0\mathcal{L}&={\rm nr}_A((z\lambda_{ij})_{1\leq i,j\leq a})e_0{\rm nr}_{A_E}(\iota^{-1})u_\calL\\\notag
&={\rm nr}_{Ae_0}(e_0(\phi\circ\delta^0)_\calF){\rm nr}_{A_Ee_0}(e_0(\iota^{-1}))u_\calL\\\notag
&=e_0{\rm nr}_A((\phi\oplus\delta^1)_\calF)u_\calL\\\notag
&=e_0{\rm nr}_A(\Phi)u_\calL\\\notag
&={\rm nr}_A(\Phi)u_\calL.
\end{align*}
Here each reduced norm is computed with respect to our fixed bases. This equality completes the proof.

\section{Non-abelian higher special elements}\label{nahse}

In this section we use the theory of non-commutative exterior powers to introduce our notion of a `non-abelian higher special element'. We also establish its basic rationality properties.

\subsection{Non-commutative exterior powers}\label{exterior}

In this section we review the construction of non-commutative exterior powers due to Burns and Sano \cite{nagm}.

The ring $A$ is semisimple and so there is a direct product Wedderburn decomposition $$A\cong\prod_{i\in I}A_i,$$ in which the index set $I$ is finite each ring $A_i$ is simple (and unique up to isomorphism).

For any choice of splitting fields $E_i$ for the rings $A_i$ over $Z(A_i)$ and of simple $E_i\otimes_{Z(A_i)}A_i$-modules $V_i$, and for any non-negative integer $r$, there is an $r$-th reduced exterior power functor ${{\bigwedge}}^r_A$ from the category of finitely generated $A$-modules to that of $Z(A)$-modules.

If $A$ is commutative then one may take $V_i=E_i=A_i$ and then this functor coincides with the standard $r$-th exterior power. In general, there are canonical choices of splitting fields $E_i$ and then, since all simple $E_i\otimes_{Z(A_i)}A_i$-modules are isomorphic, different choices of such modules $\{V_i\}_{i\in I}$ are easily seen to give naturally equivalent reduced exterior powers. See \cite[Rem. 4.4]{nagm} for more details.

Reduced exterior powers also behave well under scalar extension, in that for any algebraic extension $\mathcal{F}'$ of $\mathcal{F}$ and any finitely generated $A$-module $M$, there is an injective homomorphism ${\bigwedge}_A^rM\to{\bigwedge}_{(A\otimes_\mathcal{F}\mathcal{F}')}(M\otimes_\mathcal{F}\mathcal{F}')$.

For any $0\leq s\leq r$ there are natural duality pairings
\begin{equation}\label{dualitypairing}{\bigwedge}_A^rM\times{\bigwedge}_{A^{\rm op}}^s\Hom_A(M,A)\to{\bigwedge}_A^{r-s}M\end{equation}
for every finitely generated $A$-module $M$. In the sequel we shall denote this pairing by $(m,\varphi)\mapsto\varphi(m)$.

In addition, for fixed ordered $E_i$-bases of the spaces $V_i$, each $r$-tuples $(m_j)_{1\leq j\leq r}$ of elements of $M$ and $(\varphi_j)_{1\leq j\leq r}$ of elements of $\Hom_A(M,A)$ have associated elements $\wedge_{j=1}^{j=r}m_j$ of ${\bigwedge}_A^rM$ and $\wedge_{j=1}^{j=r}\varphi_j$ of ${\bigwedge}_{A^{\rm op}}^r\Hom_A(M,A)$. If $A$ is commutative and the required bases are specified to be the identity elements of each $V_i=E_i=A_i$ then these definitions coincide with the classical definitions of exterior products.

In general, it is proved in \cite[Lem. 4.10]{nagm} that one always has
\begin{equation}\label{2.6}(\wedge_{i=1}^{i=r}\varphi_i)(\wedge_{j=1}^{j=r}m_j)={\rm nr}_{M_r(A^{\rm op})}((\varphi_i(m_j))_{1\leq i,j\leq m})\end{equation}
so that, in particular, this element belongs to $Z(A)$ and only depends on the $r$-tuples $(m_j)_{1\leq j\leq r}$ and $(\varphi_j)_{1\leq j\leq r}$. These properties will often be useful throughout the sequel.

\begin{remark}\label{fixedconventions}{\em Let $k$ be a field with algebraic closure $k^c$ and set $G_k:={\rm Gal}(k^c/k)$.
For each irreducible complex character $\chi$ of $G_k$ that has open kernel, we write $k(\chi)$ for the subfield of $k^c$ that is fixed by $\ker(\chi)$ and $n_\chi$ for the exponent of ${\rm Gal}(k(\chi)/k)$. We also write $E_\chi$ for the field generated over $\QQ$ by a primitive $n_\chi$-th root of unity. Then there exists a representation
$$\rho_\chi:{\rm Gal}(k(\chi)/k)\to{\rm GL}_{\chi(1)}(E_\chi)$$
of character $\chi$.

For a given finite Galois extension $L/k$ in $k^c$ with Galois group $\mathcal{G}_L:={\rm Gal}(L/k)$ we write $E_L$ for the composite of the fields $E_\chi$ as $\chi$ runs over ${\rm Ir}(\mathcal{G}_L)$. Then, for a fixed choice of representations $\rho_\chi$ for each $\chi\in{\rm Ir}(\mathcal{G}_L)$, the induced homomorphisms $\rho_{\chi,*}:E_L[\mathcal{G}_L]\to M_{\chi(1)}(E_L)$ combine to give an isomorphism
$$E_L[\mathcal{G}_L]\cong\prod_{\chi\in{\rm Ir}(\mathcal{G}_L)}M_{\chi(1)}(E_L).$$

This decomposition shows that $E_L$ is a splitting field for $\QQ[\mathcal{G}_L]$, that the spaces $V_\chi:=E_L^{\chi(1)}$, considered as the first columns of the component $M_{\chi(1)}(E_L)$, are a set of representatives of the simple $E_L[\mathcal{G}_L]$-modules and that one can specify the standard basis of $E_L^{\chi(1)}$ to be the ordered basis of $V_\chi$.

In this way, the specification of a representation $\rho_\chi$ for each irreducible complex character $\chi$ of $G_k$ that has open kernel leads to a canonical choice of the data necessary to define reduced exterior powers over any algebra of the form $\mathcal{F}[{\rm Gal}(L/k)]$, with $L$ a finite Galois extension of $k$. We assume throughout the sequel and without further explicit comment that, whenever we have fixed a field $k$, all such reduced exterior powers are defined relative to such a fixed choice of representations.
}\end{remark}

\subsection{Definitions}\label{defns}

In this section we first construct the necessary idempotents and then finally define our non-abelian higher special elements.

\begin{definition}\label{eca}{\em Let $C$ be an object of $D^{\rm a}(\A)$. For any non-negative integer $a$ we define a central idempotent $e_a=e_{C,a}$ of $A$ to be the sum of all primitive central idempotents $e$ of $A$ with the property that the free $Ae$-module $eH^2(C)_\calF$ has rank $a$. We also set $e_{(a)}=e_{C,(a)}:=\sum_{b\geq a}e_b$.}\end{definition}

We now fix an object $C$ of $D^{\rm a}(\A)$ as well as a surjective homomorphism of $\A$-modules \begin{equation}\label{pi}\pi:H^2(C)\to Y_\pi.\end{equation}

\begin{definition}\label{epi}{\em We define a central idempotent $e_\pi$ of $A$ as the sum of all primitive central idempotents $e$ of $A$ for which $e(\ker(\pi)_\calF)$ vanishes.

The isomorphism
$$e_\pi\cdot\pi_\calF:e_\pi H^2(C)_\calF\to e_\pi Y_{\pi,\calF}$$ then induces, for each non-negative integer $a$, a canonical isomorphism of $Z(Ae_\pi)$-modules
$$(e_\pi\cdot\pi_\calF)^{(a)}:e_\pi(\bigwedge_A^a H^2(C)_\calF)\to e_\pi(\bigwedge_A^aY_{\pi,\calF}).$$}\end{definition}

For a given $E$-trivialisation $t\in\tau(C_E)$ of $C$, its inverse $t^{-1}$ induces, together with the map $(e_\pi\cdot\pi_\calF)^{(a)}$ for any non-negative integer $a$, a composite isomorphism $t_\pi^a$ of $Z(\A_Ee_\pi)$-modules $$e_\pi(\bigwedge^{a}_{\A_E}H^1(C)_E)\stackrel{\sim}{\to}e_\pi(\bigwedge^{a}_{\A_E} H^2(C)_E)\stackrel{\sim}{\to}e_\pi(\bigwedge^{a}_{\A_E}Y_{\pi,E}).$$

\begin{defn}\label{defn}{\em Let $\mathcal{L}$ be a characteristic element for the pair $(C,t)$ and $\pi$ be any surjective homomorphism as in (\ref{pi}).

Let $\mathcal{X}$ be a finite ordered subset of $Y_\pi$ of cardinality $a\geq 0$. The `non-abelian higher special element' associated to the data $(C,t,\mathcal{L},\pi,\mathcal{X})$ is $$\eta_{\mathcal{X}}:=(t_\pi^a)^{-1}(e_\pi\cdot e_{C,a}\cdot\mathcal{L}\cdot\wedge_{x\in\mathcal{X}}x)\in (e_\pi\cdot e_{C,a})(\bigwedge^{a}_{\A_E} H^1(C)_E).$$
}
\end{defn}

\subsection{Rationality of non-abelian higher special elements}\label{ratsection}

Before studying the finer integrality properties of non-abelian higher special elements we must establish that they live in rational non-commutative exterior powers.

\begin{lemma}\label{rationality} For any data $(C,t,\mathcal{L},\pi,\mathcal{X})$ as in Definition \ref{defn} one has $$\eta_{\mathcal{X}}\in (e_\pi\cdot e_{C,a})\bigwedge^{a}_{A}H^1(C)_\calF.$$
\end{lemma}
\begin{proof}
We set $\eta:=\eta_{\mathcal{X}}$. It is enough to prove, for each primitive central idempotent $e$ of $A$, that $e(\eta)\in (e\cdot e_\pi\cdot e_{C,a})\bigwedge^{a}_{A}H^1(C)_\calF$.

If $e\cdot e_\pi\cdot e_{C,a}=0$ then $e(\eta)=0$ and this containment is clear. So we henceforth assume that $e\cdot e_\pi\cdot e_{C,a}\neq 0$, so that $e\cdot e_\pi\cdot e_{C,a}=e$, $e(\ker(\pi)_\calF)$ vanishes and 
\begin{equation}\label{eeta}e(\eta)=(e\cdot t_\pi^a)^{-1}(e\cdot\mathcal{L}\cdot\wedge_{x\in\mathcal{X}}x)=e\cdot\mathcal{L}\cdot(e\cdot t^{(a)})((e\cdot (\pi_\calF)^{-1})^{(a)}(e\cdot\wedge_{x\in\mathcal{X}}x))\end{equation}
in $e\bigwedge^{a}_{\A_E}H^1(C)_E$.
It is enough to prove that $e(\eta)$ belongs to $ e\bigwedge^{a}_{A}H^1(C)_\calF$.

In addition, the $Ae$-module $eY_{\pi,\calF}$ is free of rank $a$ (since so is $eH^2(C)_\calF$) and, if the set $\{e\cdot x:x\in\mathcal{X}\}$ is not a basis for this module, then (\ref{2.6}) implies that $$e\cdot\wedge_{x\in\mathcal{X}}x=0$$ in $e\bigwedge_A^aY_{\pi,\calF}$. In this case also $e(\eta)=0$ and the required containment is clear. We thus assume henceforth that $\{e\cdot x:x\in\mathcal{X}\}$ is an $\A_Ee$-basis of $eY_{\pi,E}$.

Now, the $A$-modules $H^1(C)_\calF$ and $H^2(C)_\calF$ are (non-canonically) isomorphic, so we may and will fix an isomorphism $t':H^1(C)_\calF\to H^2(C)_\calF$.

After enlarging $E$ if necessary, we may assume that the reduced norm map $\nr_{\A_Ee}:K^1(\A_Ee)\to Z(\A_Ee)^\times$ is bijective, and this implies that $\ker(\delta_{Ae,E})\subseteq Z(Ae)^\times$. Since
$$\delta_{Ae,E}(e\cdot\mathcal{L})=-\chi^{\rm ref}_{Ae,E}(\A_Ee\otimes_{A}^{\mathbb{L}}C_\calF,e\cdot t)=\delta_{Ae,E}(\nr_{\A_Ee}(e\cdot t^{',-1}_E\circ e\cdot t^{-1}))$$ we thus deduce that $$e\cdot\mathcal{L}\cdot\nr_{\A_Ee}(e\cdot t\circ e\cdot t'_E)\in Z(Ae)^\times.$$
It follows that the element $$e\cdot\mathcal{L}\cdot\nr_{\A_Ee}(e\cdot t\circ e\cdot t'_E)\cdot((e\cdot t^{',-1}\circ e\cdot (\pi_\calF)^{-1})^{(a)}(e\cdot\wedge_{x\in\mathcal{X}}x))$$ belongs to $e\bigwedge^{a}_{A}H^1(C)_\calF$. Since (\ref{eeta}) combines with \cite[Lem. 4.13]{nagm} to imply that the last displayed expression is equal to $e(\eta)$, this containment completes the proof.

\end{proof}

\section{The main algebraic result}\label{mainsection}

In this section we state our main algebraic result under the specification $\A:=R[G]$. We recall that $\mathcal{D}(R[G])$ denotes the `ideal of denominators' that is introduced by Burns and Sano in \cite[Def. 3.4]{nagm}.
\subsection{The statements}

Fix any data $(C,t,\mathcal{L},\pi,\mathcal{X})$ as in Definition \ref{defn} and set $a:=|\mathcal{X}|$. For any $R[G]$-module $M$ we set $$M':=R[G]e_{C,(a)}\otimes_{R[G]}M.$$

The statement of our main result will require data of the following type:\begin{itemize}\item[(i)] an element $$x\in (R[G]\cap R[G]')\cdot\mathcal{D}(R[G]');$$ \item[(ii)] an element $$z\in{\rm Ann}_{R[G]'}(\Ext^2_{R[G]'}(H^2(C)',R[G]'));$$
\item[(iii)] an $a$-tuple of homomorphisms $$\varphi_1,\ldots,\varphi_a\in\Hom_{R[G]}(H^1(C),R[G]).$$
\end{itemize}

The proof of this result will be given in \S \ref{mainproof} below.

\begin{theorem}\label{integrality} The $R[G]'$-module $H^2(C)'$ admits a locally-free, locally-quadratic presentation $\Pi$ with the property that $${\rm nr}_{\calF[G]'}( z)^a\cdot(\wedge_{j=1}^{j=a}\varphi_j)(\eta_{\mathcal{X}})\in{\rm Fit}_{R[G]'}^a(\Pi)$$
and
$$ x\cdot{\rm nr}_{\calF[G]'}( z)^a\cdot(\wedge_{j=1}^{j=a}\varphi_j)(\eta_{\mathcal{X}})\in{\rm Ann}_{R[G]}((Y_{\pi})_{\rm tor}).$$

\end{theorem}

\begin{remark}{\em Let $C$ be a complex in $D^{\rm a}(R[G])$ which, as an strenghthening of condition (ad$_2$), satisfies the condition that the Euler characteristic of $C$ in the Grothendiek group $K_0(R[G])$ vanishes. Then one may take the presentation $\Pi$ to be free (quadratic) rather than just locally-free.
}\end{remark}

As explained in Example \ref{dual exam}, $R[G]$ is Gorenstein with respect to the anti-involution $\iota_\#$ that satisfies $\iota_\#(g) = g^{-1}$ for all $g$ in $G$, and therefore the group ${\rm Ext}^2_{R[G]}(M,R[G])$ vanishes for any finitely generated $R[G]$-module $M$.

As an immediate consequence one thus obtains the following simplification of Theorem \ref{integrality}.

\begin{corollary}\label{integralitycor} Fix $(C,t,\calL,\pi)$. Fix any non-negative integer $a$ for which the $\calF[G]$-module $\calF\otimes_R H^2(C)$ contains a free submodule of rank $a$. 

Then for any subset $\mathcal{X}$ of $Y_\pi$ of cardinality $a$, any $a$-tuple of homomorphisms $\varphi_1,\ldots,\varphi_a\in\Hom_{R[G]}(H^1(C),R[G])$ and any $x\in \mathcal{D}(R[G])$ one has $$(\wedge_{j=1}^{j=a}\varphi_j)(\eta_{\mathcal{X}})\in\Fit^a_{R[G]}(\Pi)\,\,\,\,\,\,\,\,\text{  and  }\,\,\,\,\,\,\,\,x\cdot (\wedge_{j=1}^{j=a}\varphi_j)(\eta_{\mathcal{X}})\in{\rm Ann}_{R[G]}((Y_{\pi})_{\rm tor}).$$

\end{corollary}

The following additional consequence of Theorem \ref{integrality} will also, in the sequel, lead to relatively simpler statements in our arithmetic applications.

\begin{corollary}\label{cor2} In the notation and setting of Theorem \ref{integrality}, and for any $$y\in(R[G]\cap R[G]'),$$ one has 
$${\rm nr}_{\calF[G]'}(y)^{2a}\cdot(\wedge_{j=1}^{j=a}\varphi_j)(\eta_{\mathcal{X}})\in{\rm Fit}_{R[G]'}^a(\Pi)$$
and
$$ x\cdot{\rm nr}_{\calF[G]'}(y)^{2a}\cdot(\wedge_{j=1}^{j=a}\varphi_j)(\eta_{\mathcal{X}})\in{\rm Ann}_{R[G]}((Y_{\pi})_{\rm tor}).$$
\end{corollary}
\begin{proof}
We set $M:=H^2(C)'$. To deduce the claimed result from Theorem \ref{integrality}, it is then enough to show that $y^2$ annihilates ${\rm Ext}^2_{R[G]'}(M,R[G]')$. We adapt an argument used in the proof of \cite[Thm. 8.6]{rbsd}.

To do this we use the existence of a convergent first quadrant cohomological spectral sequence
\[ E_2^{pq} = {\rm Ext}_{R[G]'}^p(M,{\rm Ext}^q_{R[G]}(R[G]',R[G]')) \Rightarrow {\rm Ext}^{p+q}_{R[G]}(M,R[G]')\]
(cf. \cite[Exer. 5.6.3]{weibel}).

In particular, since the long exact sequence of low degree terms of this spectral sequence gives an exact sequence of $R[G]$-modules
\[ \Hom_{R[G]}(M,{\rm Ext}^1_{R[G]}(R[G]',R[G]')) \to {\rm Ext}_{R[G]'}^2(M,R[G]') \to {\rm Ext}^{2}_{R[G]}(M,R[G]'),\]
we find that it is enough to show that the element $y$ annihilates both ${\rm Ext}^1_{R[G]}(R[G]',R[G]')$ and ${\rm Ext}^{2}_{R[G]}(M,R[G]')$.

To verify this we write $R[G]^\dagger$ for the ideal $\{x \in R[G]: x\cdot e_{C,(a)} = 0\}$ so that there is a natural short exact sequence of $R[G]$-modules $0 \to R[G]^\dagger \to R[G] \to R[G]' \to 0$.

Then by applying the exact functor ${\rm Ext}^\bullet_{R[G]}(-,R[G]')$ to this sequence one obtains a surjective homomorphism
\[ \Hom_{R[G]}(R[G]^\dagger,R[G]') \twoheadrightarrow {\rm Ext}^1_{R[G]}(R[G]',R[G]').\]

In addition, since as explained in Example \ref{dual exam}, $R[G]$ is Gorenstein with respect to the anti-involution $\iota_\#$ that satisfies $\iota_\#(g) = g^{-1}$ for all $g$ in $G$, by applying the exact functor ${\rm Ext}^{\bullet}_{R[G]}(M,-)$ to the above sequence one finds that there is a natural isomorphism
\[ {\rm Ext}^{3}_{R[G]}(M,R[G]^\dagger) \cong {\rm Ext}^{2}_{R[G]}(M,R[G]').\]

To complete the proof it is thus enough to note that the left hand modules in both of the last two displays are annihilated by $y$ since the definition of $R[G]^\dagger$ implies immediately that $y\cdot R[G]^\dagger = 0$.

\end{proof}

\subsection{Reduction results}

In order to prove Theorem \ref{integrality} in the next section, we begin by making the following helpful reductions.

\begin{lemma}\label{dvr} It is enough to prove Theorem \ref{integrality} in the case that $R$ is a discrete valuation ring.\end{lemma}\begin{proof} This is clear from the definition of a locally-free presentation, and of the associated higher Fitting invariants, in terms of the localised rings $R_{(\p)}[G]$ as $\p$ ranges over the prime ideals of $R$.\end{proof}

\begin{lemma}\label{ranka} Assume that $R$ is a discrete valuation ring. Then it is enough to prove Theorem \ref{integrality} in the case that the image of $\mathcal{X}$ in $Y_{\pi,\calF}'$ generates a free $R[G]'$-module of rank $a$.\end{lemma}
\begin{proof} Let $p$ denote the residue characteristic of $R$. We set $e_{\pi,a}:=e_\pi\cdot e_{C,a}$.

We label the elements of $\mathcal{X}$ as $\{x_i\}_{1\leq i\leq a}$ and fix an ordered subset $\mathcal{Y}:=\{y_i\}_{1\leq i\leq a}$ of $Y_\pi$, of cardinality $a$, that generates a free $R[G]'$-submodule of rank $a$ of $Y_{\pi,\calF}'$. By the argument of \cite[Lem. 3.16]{hse}, for any large enough integer $N$ the set $$\mathcal{Y}_N:=\{x_i+p^Ny_i\}_{1\leq i\leq a}$$ generates a free $R[G]'$-submodule of rank $a$ of $Y_{\pi,\calF}'$. For any such integer $N$ we set $\eta_N:=\eta_{\mathcal{Y}_N}$.

Now, in each simple component $\bigwedge_{\calF[G]e}^aeY_{\pi,\calF}$ of $e_{\pi,a}\bigwedge_{\calF[G]}^aY_{\pi,\calF}$, the difference
\begin{equation}\label{difference}\wedge_{i=1}^{i=a}(x_i+p^Ny_i)-\wedge_{i=1}^{i=a}x_i\end{equation} belongs to $p^N\cdot L_e$ where $L_e$ is the $R$-sublattice of $\bigwedge_{\calF[G]e}^aeY_{\pi,\calF}$ generated by $$Z:=\{\wedge_{i=1}^{i=a}z_i:z_i\in\{x_i,y_i\}\}.$$

In fact, in $e_{\pi,a}\bigwedge_{\calF[G]}^aY_{\pi,\calF}$ the difference (\ref{difference}) therefore belongs to $p^N\cdot L$ where $L$ is the $R$-sublattice generated by $Z$. Defining an $R$-sublattice $$L':=e_{\pi,a}\cdot\calL\cdot(t_\pi^a)^{-1}(L)$$ of $e_{\pi,a}\bigwedge_{\calF[G]}^aH^1(C)_\calF$ one then finds that $$\eta_N-\eta\in p^N\cdot L'.$$

We next define an $R$-sublattice $$L'':=\{(\wedge_{j=1}^{j=a}\varphi_j)(L'):\varphi_\bullet\in\Hom_{R[G]}(H^1(C),R[G])^a\}$$ of $Z(\calF[G]e_{\pi,a})\subseteq Z(\calF[G]')$, so that
\begin{equation*}\{(\wedge_{j=1}^{j=a}\varphi_j)(\eta_N)-(\wedge_{j=1}^{j=a}\varphi_j)(\eta):\varphi_\bullet\in\Hom_{R[G]}(H^1(C),R[G])^a\} \subseteq p^N\cdot L''.\end{equation*}

We note that, for any large enough choice of $N$, one has \begin{align*}(R[G]\cap R[G]')\cdot\mathfrak{D}(R[G]')\cdot p^N\cdot L''\subseteq&(R[G]\cap R[G]')\cdot{\rm Ann}_{R[G]'}((Y'_\pi)_{\rm tor})\\ \subseteq&{\rm Ann}_{R[G]}((Y_\pi)_{\rm tor}),\end{align*}
where the first inclusion holds because $(Y'_\pi)_{\rm tor}$ is finite and the second inclusion follows from (\ref{intersection}) below.

It is now enough to prove that, for any large enough choice of $N$ and any free, quadratic $R[G]'$-presentation $\Pi$ of $H^2(C)'$ one has \begin{equation*}\label{enoughfit}p^N\cdot L''\subseteq\Fit_{R[G]'}^a(\Pi),
\end{equation*} and this would follow from the fact that $\Fit_{R[G]'}^a(\Pi)\cap Z(\calF[G]e_{\pi,a})$ has finite index in $Z(\calF[G]e_{\pi,a})$. It is thus enough to prove that \begin{equation*}\label{enoughfit}e_{\pi,a}\cdot(\calF\otimes_R\Fit_{R[G]'}^a(\Pi))=Z(\calF[G]e_{\pi,a})\end{equation*} and this is what we proceed to do.

The presentation $\Pi$ is of the form $$R[G]^{',d}\stackrel{\theta}{\to}R[G]^{',d}\to H^2(C)'\to0$$ and induces a free presentation of $\calF[G]e_{\pi,a}$-modules
$$(\calF[G]e_{\pi,a})^d\stackrel{e_{\pi,a}\cdot\theta_\calF}{\to}(\calF[G]e_{\pi,a})^d\stackrel{\varpi}{\to} e_{\pi,a}\cdot H^2(C)_\calF\to 0.$$ 
It is enough to prove that, for some matrix $B\in M_d(\calF[G]e_{\pi,a})$ that coincides with $M_{e_{\pi,a}\cdot\theta_\calF}$ in all but $a$ columns, the reduced norm of $B$ is a unit in $Z(\calF[G]e_{\pi,a})$. Here $$M_{e_{\pi,a}\cdot\theta_\calF}\in M_d(\calF[G]e_{\pi,a})$$ denotes the matrix of $e_{\pi,a}\cdot\theta_\calF$ with respect to any choice of $\calF[G]e_{\pi,a}$-bases.

But the $\calF[G]e_{\pi,a}$-module $e_{\pi,a}\cdot H^2(C)_\calF$ is free of rank $a$. Fixing a section to $\varpi$ induces a direct sum decomposition $$(\calF[G]e_{\pi,a})^d=\im(e_{\pi,a}\cdot\theta_\calF)\oplus e_{\pi,a}\cdot H^2(C)_\calF$$ where $\im(e_{\pi,a}\cdot\theta_\calF)$ is free of rank $d-a$. It is now clear that one may find bases with respect to which $M_{e_{\pi,a}\cdot\theta_\calF}$ is a block matrix of the form
$$\left(
\begin{array}{c|c}
\begin{array}{c}I_{d-a}\\ 0\end{array} & \star
\end{array}
\right).$$
The identity matrix in $M_d(\calF[G]e_{\pi,a})$ therefore coincides with 
such a choice $M_{e_{\pi,a}\cdot\theta_\calF}$ in all but the last $a$ columns, and is thus a valid choice of matrix $B$. This fact completes the proof.
\end{proof}

\subsection{The proof of Theorem \ref{integrality}}\label{mainproof}

In this section we complete the proof of Theorem \ref{integrality}.
In view of Lemmas \ref{dvr} and \ref{ranka}, in the rest of this proof we may and will assume that $R$ is a discrete valuation ring and that the image of $\mathcal{X}$ in $Y_{\pi,\calF}'$ generates a free $R[G]'$-module $X$ of rank $a$.

We consider the object $$C':=R[G]e_{(a)}\otimes_{R[G]}^{\mathbb{L}}C$$ of $D^{\rm p}(R[G]')$. We fix, as we may, an (injective) $R[G]'$-lift $\iota_2:X\to H^2(C)'=H^2(C')$ of the inclusion $X\subseteq Y_\pi'$ via $\pi'=e_{(a)}\pi$. We also fix any injective homomorphism $\iota_1:X\to H^1(C')$.

Since $R$ is a discrete valuation ring and $C$ is an admissible complex of $R[G]$-modules, a standard argument in homological algebra implies that one may fix a representative of $C$ of the form $$F^1\stackrel{\partial}{\to}F^2,$$ where $F^1$ and $F^2$ are finitely generated free $R[G]$-modules of the same rank and $F^1$ is placed in degree one.

We fix any such choice of explicit representative for $C$ and associate to it an explicit complex $$D:X\stackrel{0\oplus\iota_1}{\longrightarrow}X\oplus F^{1,'}\stackrel{(\iota_2,\partial')}{\longrightarrow}F^{2,'}$$ of $R[G]'$-modules. Here the first term is placed in degree zero, $\partial'=e_{(a)}\partial$ and, by abuse of notation, $\iota_1$ is the composition $$X\stackrel{\iota_1}{\to}H^1(C')=\ker(\partial')\subseteq F^{1,'}$$ while $\iota_2:X\to F^{2,'}$ is any lift of $\iota_2$ via the canonical surjection $F^{2,'}\to\cok(\partial')=H^2(C')$.

By \cite[Lem. 3.22]{hse}, modifying our given $a$-tuple of homomorphisms $\varphi_1,\ldots,\varphi_a$ by setting $\varphi'_j:=e_{(a)}\varphi_j$ gives a well-defined $a$-tuple of homomorphisms in $\Hom_{R[G]'}(H^1(C'),R[G]')$. After setting $\mathcal{X}=\{x_1,\ldots,x_a\}$ one then computes
\begin{align}\label{compute}\notag
(\wedge_{j=1}^{j=a}\varphi_j)(\eta_{\mathcal{X}})=&(\wedge_{j=1}^{j=a}\varphi'_j)(\eta_{\mathcal{X}})\\
=&\nr_{\calF[G]'}\left(\left(\varphi'_j(\iota_1(x_i))\right)_{1\leq i,j\leq a}\right)\cdot e_a\cdot\calL\cdot\nr_{E[G]'}(\iota_{1,E}^{-1}\circ e_{(a)}t\circ e_{(a)}\pi_E^{-1}).
\end{align}
This computation uses the equality (\ref{2.6}) and the result \cite[Lem. 4.13]{nagm}, as well as the linearity of the duality pairing (\ref{dualitypairing}).

We may now apply Corollary \ref{omaccor} to the complex of $R[G]'$-modules $D$, the $a$-tuple of homomorphisms $\varphi'_1,\ldots,\varphi'_a$, the element $z$ of $${\rm Ann}_{R[G]'}(\Ext^2_{R[G]'}(H^2(C)',R[G]'))\cong {\rm Ann}_{R[G]'}(\Ext^2_{R[G]'}(H^2(D),R[G]'))$$ and the characteristic element $\calL_D$ for $D$ provided by Lemma \ref{below} below. In order to do so, we fix $R[G]$-bases $b^1_\bullet$ and $b^2_\bullet$ of $F^1$ and $F^2$ respectively, write $\Delta$ for the matrix of $\partial$ as computed with respect to these bases and write $\Gamma$ for the matrix of $\iota_2:X\to F^{2,'}$ as computed with respect to $\mathcal{X}$ and $e_{(a)}\cdot b^2_\bullet$.

These results then combine with (\ref{compute}) to imply that there exists a canonical free presentation $\Pi_D$ of the $R[G]'$-module $H^2(D)=H^2(C)'/\iota_2(X)$ with the property that
\begin{equation}\label{fitintersection}\nr_{\calF[G]'}(z)^a\cdot(\wedge_{j=1}^{j=a}\varphi_j)(\eta_{\mathcal{X}})\in\left(\Fit_{R[G]'}^{0,{\rm tot}}(\Pi_D)\cap\Fit_{R[G]'}^a\left( 
\left(
\begin{array}{c|c}
\begin{array}{c}I_a\\ 0\end{array} & \begin{array}{c}\Gamma\\ e_{(a)}\Delta\end{array}
\end{array}
\right)
\right)\right).
\end{equation}

We next construct the free quadratic presentation $\Pi$ of $H^2(C)'$ that is claimed to exist in Theorem \ref{integrality}.
We define
$$\theta:X\oplus F^{1,'}\to X\oplus F^{2,'}$$by setting $$\theta(\alpha,\beta):=(\alpha,\iota_2(\alpha)+\partial'(\beta)).$$
If we then define a map $$\rho:X\oplus F^{2,'}\to \cok(\partial')=H^2(C)'$$ by setting $$\rho(\gamma,\delta):=(-\iota_2(\gamma)+\delta)+\im(\partial')$$ then the sequence
$$\Pi:X\oplus F^{1,'}\stackrel{\theta}{\to} X\oplus F^{2,'}\stackrel{\rho}{\to}H^2(C)'\to 0$$ is exact. It is in addition clear that
$$\Fit_{R[G]'}^a\left( 
\left(
\begin{array}{c|c}
\begin{array}{c}I_a\\ 0\end{array} & \begin{array}{c}\Gamma\\ e_{(a)}\Delta\end{array}
\end{array}
\right)
\right)=\Fit_{R[G]'}^a(\Pi).$$

The proof of Theorem \ref{integrality} is now completed upon combining (\ref{fitintersection}) with the following result.

\begin{lemma} For any free presentation $\Pi_D$ of the $R[G]'$-module $H^2(C)'/\iota_2(X)$ one has
$$(R[G]\cap R[G]')\cdot\mathcal{D}(R[G]')\cdot\Fit_{R[G]'}^{0,{\rm tot}}(\Pi_D)\subseteq{\rm Ann}_{R[G]}((Y_\pi)_{\rm tor}).$$
\end{lemma}
\begin{proof}One has that  $\mathcal{D}(R[G]')\cdot\Fit_{R[G]'}^{0,{\rm tot}}(\Pi_D)$ is contained in
${\rm Ann}_{R[G]'}(H^2(C)'/\iota_2(X))$ by Lemma \ref{BSlemma}. Since, by choice of $\iota_2$, the surjective map $\pi$ induces a surjection
$$H^2(C)'/\iota_2(X)=(R[G]'\otimes_{R[G]}H^2(C))/\iota_2(X)\to(R[G]'\otimes_{R[G]}Y_\pi)/X,$$ and since $X$ is a free $R[G]'$-module, we find that $\mathcal{D}(R[G]')\cdot\Fit_{R[G]'}^{0,{\rm tot}}(\Pi_D)$ is also contained in ${\rm Ann}_{R[G]'}((R[G]'\otimes_{R[G]}Y_\pi)_{\rm tor})$.

To conclude the proof, it is thus enough to note that \begin{equation}\label{intersection}(R[G]\cap R[G]')\cdot{\rm Ann}_{R[G]'}((R[G]'\otimes_{R[G]}Y_\pi)_{\rm tor})\subseteq{\rm Ann}_{R[G]}((Y_\pi)_{\rm tor}).\end{equation} Indeed, if we set $R[G]^\#:=R[G]\cap R[G](1-e_{(a)})$, the tautological exact sequence $0\to R[G]^\#\to R[G]\to R[G]'\to 0$ gives rise to an exact sequence of $R[G]$-modules $$R[G]^\#\otimes_{R[G]}Y_{\pi}\to Y_\pi\to R[G]'\otimes_{R[G]}Y_\pi\to 0.$$ The required inclusion then follows from the fact the first term in this sequence is annihilated by $R[G]\cap R[G]'$ and the fact that any element in the kernel of the induced map $$(Y_{\pi})_{\rm tor}\to(R[G]'\otimes_{R[G]}Y_\pi)_{\rm tor}$$ is contained in the image of this first term.
\end{proof}

We finally prove the intermediate result that was used in the course of the proof of Theorem \ref{integrality}

\begin{lemma}\label{below} There exists a characteristic element $\calL_D\in Z(E[G]')^\times$ for $D$ with the property that
$$e_a\cdot\calL_D=e_a\cdot\calL\cdot\nr_{E[G]'}(\iota_{1,E}^{-1}\circ e_{(a)}t\circ e_{(a)}\pi_E^{-1}).$$
\end{lemma}
\begin{proof}It is easy to see that there is a canonical exact triangle in $D^{\rm p}(R[G]')$ of the form
\begin{equation*}\label{key triangle} X[-2] \oplus X[-1] \xrightarrow{\iota} C'\to D \to X[-1] \oplus X[0]\end{equation*}
in which $H^1(\iota) = \iota_1$ and $H^2(\iota) = \iota_2$.

Since $t^{-1}$  induces an isomorphism of $E[G]$-modules between $e_a H^1(C)_E=e_a\iota_1(X)_E$ and $e_a H^2(C)_E= e_a \iota_2(X)_E$ we can fix a commutative diagram of $E[G]'$-modules
\begin{equation*}\label{key diagram} \begin{CD}
0 @> >> X_E @>\iota_{1,E}  >> H^1(C')_E @> >> H^1(D)_E @> >> 0\\
@. @V t_1 VV @V t_2 VV @V t_3 VV\\
0 @> >> X_E @> \iota_{2,E} >> H^2(C')_E @> >> H^2(D)_E @> >> 0\end{CD}\end{equation*}
where the maps $t_1$, $t_2$ and $t_3$ are bijective and satisfy
\begin{equation}\label{restrictions agree} (e_a\pi_E)^{-1}\circ e_a t_1\circ e_a\iota_{1,E}^{-1}=e_a(\iota_{2,E}\circ t_1\circ\iota_{1,E}^{-1})=e_at_2=e_at^{-1}.\end{equation}

We consider the object $$D_a:=R[G]e_a\otimes_{R[G]'}^{\mathbb{L}}D$$ of $D^{\rm p}(R[G]e_a)$. Then, since the definition of $e_a=e_{a,C}$ implies that, in the notation of \S \ref{idempotents}, it is equal to $e_0(D)=e_{D,0}$, Lemma \ref{lam lemma} implies that it is enough to prove that
\begin{equation}\label{enoughcharD}-\chi^{\rm ref}_{R[G]e_a,E}(D_a, e_at_3^{-1})=\delta_{R[G]e_a,E}(e_a\cdot\mathcal{L}\cdot \nr_{E[G]'}(\iota_{1,E}^{-1}\circ e_{(a)}t\circ e_{(a)}\pi_E^{-1}))\end{equation}
in $K_0(R[G]e_a,E[G]e_a)$.
Since one has
\[\chi^{\rm ref}_{R[G]',E}(C', t_2^{-1})=\chi^{\rm ref}_{R[G]',E}(D, t_3^{-1})+\chi^{\rm ref}_{R[G]',E}(X[-2] \oplus X[-1], t_1^{-1})\]
in $K_0(R[G]',E[G]')$, the required equality (\ref{enoughcharD}) now follows from the following explicit computation that uses (\ref{restrictions agree}):
\begin{align*}-\chi^{\rm ref}_{R[G]e_a,E}(D_a, e_at_3^{-1})
=&\chi^{\rm ref}_{R[G]e_a,E}(e_aX[-2] \oplus e_aX[-1], e_a\iota_{1,E}^{-1}\circ e_at\circ (e_a\pi_E)^{-1})\\
& -\chi^{\rm ref}_{R[G]e_a,E}(R[G]e_a\otimes_{R[G]}^{\mathbb{L}}C,e_at)\\
=&[e_aX,e_aX,e_a\iota_{1,E}^{-1}\circ e_at\circ (e_a\pi_E)^{-1}]+\delta_{R[G]e_a,E}(e_a\calL)\\ 
=& \delta_{R[G]e_a,E}(e_a\cdot\mathcal{L}\cdot \nr_{E[G]'}(\iota_{1,E}^{-1}\circ e_{(a)}t\circ e_{(a)}\pi_E^{-1})).\end{align*}

\end{proof}

\section{Selmer and Tate-Shafarevich groups and refined BSD conjectures}\label{ts}

In this section we extend the result \cite[Thm. 8.6]{rbsd} of Burns and the first author, concerning the Galois structures of Selmer and Tate-Shafarevich groups of abelian varieties, and their relations to the formulation of refined conjectures of Birch and Swinnerton-Dyer type, from the setting of abelian extensions of number fields to the setting of general Galois extensions.

In \S \ref{dihedralsect} we will then consider dihedral twists of elliptic curves over general number fields. By combining our general approach with a result of Mazur and Rubin \cite[Thm. B]{mr2} we are able to obtain strikingly explicit predictions for the first derivatives of Hasse-Weil-Artin $L$-series of such twists. In Example \ref{heegner} below we discuss how further specialisation leads to conjectural relationships between the arithmetic of `higher Heegner points' in (generalised) dihedral extensions of $\QQ$ and the Galois module structure of Tate-Shafarevich and Selmer groups.

\subsection{The general case}\label{8.1}

Let $F/k$ be a finite Galois extension of number fields with Galois group $G$. We set $n := [k:\QQ]$.

Let $A$ be an abelian variety of dimension $d$ defined over $k$. We write $A^t$ for the dual abelian variety and in general use the notation introduced in \S \ref{selmerexample} and in \S \ref{bkexample}. In particular, we always assume in the sequel that the Tate-Shafarevich group $\sha(A_F)$ of $A$ over $F$ is finite.

We recall that Burns and the first author have recently formulated a general `refined conjecture of Birch and Swinnerton-Dyer type' \cite[Conj. 3.3]{rbsd}, and that this conjecture decomposes naturally into $p$-primary components for each rational prime $p$. In the sequel we fix an odd prime number $p$ and refer to the $p$-primary component of this conjecture as ${\rm BSD}_p(A_{F/k})$. See Lemma 3.11 in loc. cit. for a precise  statement of the conjecture ${\rm BSD}_p(A_{F/k})$.

Throughout this section we give ourselves a fixed odd prime $p$ and an isomorphism of fields $\CC\cong\CC_p$ (that we will use to make certain implicit identifications but usually avoid mentioning explicitly).
We recall that the N\'eron-Tate height pairing for $A$ over $F$ combines with this isomorphism to induce a canonical isomorphism of $\CC_p[G]$-modules
\[ h_{A_{F/k}}: A^t(F)_{\CC_p} \cong \Hom_{\CC_p}( A(F)_{\CC_p},\CC_p) =  \Hom_{\ZZ_p[G]}(A(F)_p,\ZZ_p[G])_{\CC_p}.\]

The formulation of the conjecture ${\rm BSD}_p(A_{F/k})$ relies on fixing a finite set $S$ of places of $k$ as well as a basis $\omega_\bullet$ of global differentials. Its validity, however, is independent of these choices (see \cite[Rem. 3.9(i)]{rbsd}).

We thus fix a finite set $S$ of places of $k$ with
\begin{equation}\label{starSset} S_\infty(k)\cup S_p(k)\cup  S_{\rm ram}(F/k) \cup S_{\rm bad}(A)\subseteq S.\end{equation}


We also fix an ordered $k$-basis $\{\omega'_j: j \in [d]\}$ of $H^0(A^t,\Omega^1_{A^t})$ and we use this basis to define an explicit `classical period' $\Omega_A^{F/k}$ in $Z(\CC[G])^{\times}$ as in \cite[(42)]{rbsd}.

We also denote by $w_{F/k}$ the `root number' defined in \cite[(43)]{rbsd}.


In order to state the main result of this section we need to introduce some additional preliminary notation.

\subsubsection{Logarithmic resolvents}

We set $F_p:=\prod_{w\mid p}F_w$ and for each index $j$ we then write ${\rm log}_{\omega'_j}:=\prod_{w\mid p}{\rm log}_{F_w,\omega_j'}$ for the formal logarithm of $A^t$ over $F_p$ that is defined with respect to $\omega_j'$.

We also fix an ordering of the set $\Sigma(k)$ of embeddings $k\to\CC$. We write $\CC_p[G]^{nd}$ for the direct sum of $nd$ copies of $\CC_p[G]$ and fix a bijection between the standard basis of this module and the lexicographically-ordered direct product $\{1,\ldots,d\}\times \Sigma(k)$.

Then for any ordered subset
\[ x_\bullet:= \{x_{(i,\sigma)}: (i,\sigma) \in \{1,\ldots,d\}\times\Sigma(k)\}\]
of $A^t(F_p)^\wedge_p$ we define a logarithmic resolvent element of $Z(\CC_p[G])$ by setting
\[ \mathcal{LR}^p_{A^t_{F/k}}(x_\bullet) := {\rm nr}_{\QQ^c_p[G]}\left(\bigl(\sum_{g \in G} \hat \sigma(g^{-1}({\rm log}_{\omega_j'}(x_{(j',\sigma')})))\cdot g \bigr)_{(j,\sigma),(j',\sigma')}\right) \]
where the indices $(j,\sigma)$ and $(j',\sigma')$ run over $\{1,\ldots,d\}\times \Sigma(k)$, $\hat\sigma$ is the scalar extension to $F_p$ of a fixed extension to $F$ of $\sigma\in\Sigma(k)$, and ${\rm nr}_{\QQ^c_p[G]}(-)$ denotes the reduced norm of the given matrix in ${\rm M}_{dn}(\QQ_p^c[G])$.

\subsubsection{Higher derivatives of Hasse-Weil-Artin $L$-series}

For each $\psi\in{\rm Ir}(G)$ we write $L_S(A,\psi,z)$ for the Hasse-Weil-Artin $L$-function of $A$ and $\psi$, truncated by removing the Euler factors corresponding to places in $S$.

We always assume that this function has an analytic continuation to $z=1$, where it has a zero of order equal to the multiplicity $r(A,\psi)$ with which the character $\psi$ occurs in the representation $\CC\cdot A^t(F)$ of $G$, as is conjectured by Deligne and Gross (cf. \cite[p. 127]{delignegross}).

For each non-archimedean place $v$ of $k$ that does not ramify in $F/k$ and at which $A$ has good reduction we define an element of $Z(\QQ [G])$ by setting

\[ P_v(A_{F/k},1) := {\rm nr}_{\QQ_p[G]}\bigl(1-\Phi_v\cdot a_v +  \Phi_v^2\cdot {\rm N}v^{-2}\bigr).\]
Here $\Phi_v\in G$ denotes the Frobenius automorphism of (a fixed place of $F$ above) $v$ while ${\rm N}v$ denotes its absolute norm and $a_v$ is the integer $1 + {\rm N}v - |\tilde A(\kappa_v)|$ where $\tilde A$ is the reduction of $A$ to the residue field $\kappa_v$ of $v$.

For a non-negative integer $a$ we write ${\rm Ir}(G)_{A,(a)}$ for the subset of ${\rm Ir}(G)$ comprising characters $\psi$ for which $r(A,\psi)\geq a$. This definition ensures that the $Z(\CC[G])$-valued function
\[ L^{(a)}_{S}(A_{F/k},z) := \sum_{\psi \in {\rm Ir}(G)_{A,(a)}}z^{-a}L_S(A,\check\psi,z)\cdot e_\psi\]
is holomorphic at $z=1$.

\subsubsection{Idempotents}

For each $a$ we also define idempotents in $Z(\QQ[G])$ by setting
\[ e_{(a)} = e_{F,(a)} := \sum_{\psi \in {\rm Ir}(G)_{A,(a)}}e_\psi\]
and
\[ e_{a} = e_{F,a}:= \sum_{\psi \in {\rm Ir}(G)_{A,(a)}\setminus {\rm Ir}(G)_{A,(a+1)}}e_\psi\]
(so that $e_{(a)} = \sum_{b \ge a}e_b$). We also set $R_{(a)}:=\ZZ_p[G]e_{(a)}$ and $I_{(a)}:=\ZZ_p[G]\cap R_{(a)}$.

%



\subsubsection{The main result}


The proof of this result will be given in \S\ref{proof of big conj} below.

\begin{theorem}\label{big conj} 
Fix an ordered maximal subset $x_\bullet:= \{x_{(i,\sigma)}: (i,\sigma) \in \{1,\ldots,d\}\times\Sigma(k)\}$ of $A^t(F_p)^\wedge_p$ that is linearly independent over $\ZZ_p[G]$ and a finite non-empty set $T$ of places of $k$ that is disjoint from $S$. 

If the conjecture ${\rm BSD}_p(A_{F/k})$ of \cite{rbsd} is valid, then for any non-negative integer $a$, for
any $a$-tuples $\theta_\bullet$ in $\Hom_{\ZZ_p[G]}(A^t(F)_p,\ZZ_p[G])$ and $\phi_\bullet$ in $\Hom_{\ZZ_p[G]}(A(F)_p,\ZZ_p[G])$, any element $\alpha$ of $I_{(a)}\cdot\mathcal{D}(R_{(a)})$ and any element $y$ of $I_{(a)}$, the product
\begin{multline}\label{key product} \alpha
\cdot{\rm nr}_{\QQ_p[G]e_{(a)}}(y)^{2a}\cdot(\prod_{v \in T}\iota_\#(P_v(A_{F/k},1)))\\ \cdot  \frac{L^{(a)}_{S}(A_{F/k},1)}{\Omega_A^{F/k}\cdot w_{F/k}^d}\cdot \mathcal{LR}^p_{A^t_{F/k}}(x_\bullet)\cdot (\wedge_{j=1}^{j=a}\theta_j)({\rm ht}^{(a)}_{A_{F/k}}(\wedge_{i=1}^{i=a}\phi_i))\end{multline}
belongs to $\ZZ_p[G]$ and annihilates $\sha(A^t_{F})[p^\infty]$. 
\end{theorem}

\begin{remark}{\em If $C\in D^{\rm a}(\ZZ_p[G])$ is the Nekov\'a\v r-Selmer complex (associated to $x_\bullet$ and $T$) constructed in \cite[Lem. 8.13 (ii)]{rbsd}, then our methods will show that the $R_{(a)}$-module $R_{(a)}\otimes_{\ZZ_p[G]}H^2(C)$ admits a free, quadratic presentation $\Pi$ with the property that, if ${\rm BSD}_p(A_{F/k})$ is valid, then the product (\ref{key product}) with the term $\alpha$ omitted,
belongs to ${\rm Fit}^a_{R_{(a)}}(\Pi)$.
}\end{remark}

We remark on several ways in which Theorem \ref{big conj} either simplifies or becomes more explicit.

\begin{remark}\label{more explicit rem}\label{Temptyset}{\em 

If $A(F)$ does not contain an element of order $p$, then our methods will show that the prediction in Theorem \ref{big conj} should remain true if the term $\prod_{v \in T}\iota_\#(P_v(A_{F/k},1))$ is omitted from the product (\ref{key product}).  Indeed, in this case one may apply Corollary \ref{cor2} directly to the Nekov\'a\v r-Selmer complex that occurs in Proposition \ref{prop:perfect2} (iii) rather than to its $T$-modification.


}\end{remark}

\begin{remarks}\label{e=1 case}{\em \

\noindent{}(i) In special cases one can either show, or is led to predict, that the idempotent $e_{(a)}$ belongs to $\ZZ_p[G]$ and hence that the term $$\alpha\cdot{\rm nr}_{\QQ_p[G]e_{(a)}}(y)^{2a}$$ in the product (\ref{key product}) can be taken to be any element of $\mathcal{D}(R_{(a)})$.

This is, for example, the case if $a = 0$, since each function $L(A, \psi,z)$ is holomorphic at $z=1$ and thus $e_{(0)}=1$. This situation can also arise naturally in cases with $a=1$ thanks to the existence of Heegner points or, more generally, to the result of Mazur and Rubin in \cite[Th. B]{mr2}. We shall consider the latter cases in detail in \S \ref{dihedralsect} below.

\noindent{}(ii) If in addition $e_{(a)} = 1$, there exist $a$-tuples in $A^t(F)$ and $A(F)$ that are each linearly independent over $\QQ[G]$ and this fact implies the expressions $(\wedge_{j=1}^{j=a}\theta_j)({\rm ht}^{(a)}_{A_{F/k}}(\wedge_{i=1}^{i=a}\phi_i))$ in Theorem \ref{big conj} can be interpreted in terms of classical N\'eron-Tate heights.

To be a little more precise we use the following notation: for ordered $a$-tuples $P_\bullet = \{P_i: i \in [a]\}$ of $A^t(F)_p$ and $Q_\bullet = \{Q_j: j \in [a]\}$ of $A(F)_p$  we define a matrix in ${\rm M}_a(\CC_p[G])$ by setting
\begin{equation*}\label{regulatormatrix} h_{F/k}(P_\bullet, Q_\bullet) := (\sum_{g \in G}\langle g(P_i),Q_j\rangle_{A_F}\cdot g^{-1})_{1\le i,j\le a},\end{equation*}
where $\langle -,-\rangle_{A_F}$ denotes the Neron-Tate height pairing for $A$ over $F$ (and we have again used our fixed isomorphism $\CC\cong\CC_p$).

Then a direct generalisation of the argument used to prove \cite[Lem. 8.10]{rbsd} gives a direct relationship between the reduced norm of the matrix $e_a\cdot h_{F/k}(P_\bullet, Q_\bullet)$ of ${\rm GL}_a(\CC_p[G]e_a)$ 
and the set
$$\xi(\ZZ_p[G]e_a)\cdot\left\{(\wedge_{j=1}^{j=a}\theta_j)({\rm ht}^{(a)}_{A_{F/k}}(\wedge_{i=1}^{i=a}\phi_i))\mid \theta_\bullet\subset A^t(F)_p^*,\phi_\bullet\subset A(F)_p^*\right\}.$$

In a related direction, Theorem \ref{dihedralthm} (ii) below also gives a similar explicit computation in cases with $a=1$ but without necessarily assuming that $e_{(1)}=1$.

\noindent{}(iii) Finally we note that the result \cite[Prop. 8.11]{rbsd} allows one to explicitly compute the relevant logarithmic resolvents occurring in (\ref{key product}) in many situations of interest.

}
\end{remarks}

\begin{remark}\label{new remark SC}{\em Under suitable additional hypotheses it is also possible to obtain considerably more explicit versions of the containments predicted by Theorem \ref{big conj}, avoiding the use of logarithmic resolvents, which we expect should be amenable to numerical testing in (non-abelian) examples.

To be more precise, assume that neither $A(F)$ nor $A^t(F)$ has a point of order $p$, that $p$ is unramified in $k$, that all $p$-adic places of $k$ are at most tamely ramified in $F$ and that $A$, $F/k$ and $p$ satisfy the hypotheses (H$_1)$-(H$_5$) that are listed in \S\ref{bkexample}.

Then, after taking account of the equality in \cite[Rem. 6.6]{rbsd}, the argument that is used to prove Theorem \ref{big conj} can be directly applied to the classical Selmer complex ${\rm SC}_p(A_{F/k})$ rather than to the Nekov\'a\v r-Selmer complex associated to $S$ and to our choice of semi-local points $x_\bullet$. 

In any such situations, one finds that the $R_{(a)}$-module $R_{(a)}\otimes_{\ZZ_p[G]}\Sel_p(A_F)^\vee$ admits a free, quadratic presentation $\Pi_{\Sel}$, and that ${\rm BSD}_p(A_{F/k})$ predicts that for any given non-negative integer $a$ and any data as in Theorem \ref{big conj}, the product
\begin{equation}\label{key product2} {\rm nr}_{\QQ_p[G]e_{(a)}}(y)^{2a}\cdot\frac{L^{(a)}_{S_{{\rm ram}}}(A_{F/k},1)}{\Omega_A^{F/k}\cdot w_{F/k}^d}\cdot(\tau^*(F/k)\cdot\prod_{v\in S_{p,{\rm ram}}}\varrho_v)^d\cdot (\wedge_{j=1}^{j=a}\theta_j)({\rm ht}^{(a)}_{A_{F/k}}(\wedge_{i=1}^{i=a}\phi_i))\end{equation}
should belong to ${\rm Fit}^{a}_{R_{(a)}}(\Pi_{\Sel})$, and then also to ${\rm Ann}_{\ZZ_p[G]}(\sha(A_F^t)[p^\infty])$ after multiplication by any element $\alpha$ in $I_{(a)}\cdot\mathcal{D}(R_{(a)})$.

Here $L^{(a)}_{S_{{\rm ram}}}(A_{F/k},1)$ is as defined above but with each $L$-function truncated only at the set of non-archimedean places $S_{{\rm ram}}(F/k)$ which ramify in $F/k$ rather than at all places in $S$ (as in the expression (\ref{key product})). In addition, $\tau^*(F/k)$ is the (modified) global Galois-Gauss sum of $F/k$ defined in \cite[\S 4.2.1]{rbsd}, we have used the notation $S_{p,{\rm ram}}:=S_p(k)\cap S_{{\rm ram}}(F/k)$ and, for each $v$ in this intersection, we have also set
$$\varrho_v:=\sum_{\psi\in{\rm Ir}(G)}\det({\rm N}v\mid V_\psi^{I_v})\cdot e_\psi.$$
Here ${\rm N}v$ denotes the absolute norm of $v$, $I_v$ is the inertia subgroup of $v$ in $G$ and $V_\psi$ is any fixed complex representation of $G$ of character $\psi$.

This special case of Theorem \ref{big conj} is thus itself a strong generalisation and refinement of \cite[Prop. 5.2]{bmw} and of \cite[Cor. 3.5 (ii)]{dmc}. We shall also make products of the form (\ref{key product2}) fully explicit for dihedral twists of elliptic curves in Theorem \ref{dihedralthm}(iii) below.
}\end{remark}

\begin{remark}\label{integrality rk}{\em One may obtain concrete congruences from the integrality claims of Theorem \ref{big conj} or its variant in Remark \ref{new remark SC} exactly as in \cite[Pred. 8.5]{rbsd}. In this regard see also Theorem \ref{dihedralthm}(iii) below.
%
%
} \end{remark}

\subsubsection{The proof of Theorem \ref{big conj}}\label{proof of big conj}

Fix an ordered subset $x_\bullet$ of $A^t(F_p)^\wedge_p$ as well as a set $T$ of places of $k$ as in Theorem \ref{big conj}. Write $X$ for the $\ZZ_p[G]$-module generated by $x_\bullet$.

Then we shall consider the $T$-modified Nekov\'a\v r-Selmer complex $C_{S,X,T}\in D^{\rm a}(\ZZ_p[G])$ that is constructed in \cite[Lem. 8.13 (ii)]{rbsd}, together with the $\CC_p$-trivialisation $(h^T_{A,F})^{-1}$ and, under the assumption that ${\rm BSD}_p(A_{F/k})$ is valid, characteristic element $\calL_T\in\zeta(\CC_p[G])^\times$, that are constructed in \cite[Lem. 8.13 (iii)]{rbsd}.
We recall that for any non-negative integer $a$ one has

%
\[ e_a\cdot\mathcal{L}_T=(\prod_{v \in T}\iota_\#(P_v(A_{F/k},1))) 
\cdot\frac{L^{(a)}_{S}(A_{F/k},1)}{\Omega_A^{F/k}\cdot w_{F/k}^d}\cdot \mathcal{LR}^p_{A^t_{F/k}}(x_\bullet).\]
(We note that a different normalisation in the notion of a characteristic element justifies the disparities in sign between the above equality and the one occurring in \cite[Lem. 8.13 (iii)]{rbsd}).

We also recall that from \cite[Lem. 8.13 (ii)]{rbsd} that there exists a surjective homomorphism
$$\pi:H^2(C_{S,X,T})\to Y_\pi$$ with finite kernel and the property that $Y_\pi$ contains $\Sel_p(A_F)^\vee$ as a submodule of finite index. There also exists a canonical injective homomorphism $\iota:H^1(C_{S,X,T})\to A^t(F)_p$ with finite cokernel.


We will apply Corollary \ref{cor2} to the triple given by $(C_{S,X,T},(h_{A,F}^T)^{-1},\mathcal{L}_T)$, together with the surjective homomorphism $\pi$.

In order to do so, we fix a non-negative integer $a$ and an $a$-tuple $\phi_\bullet$ in $\Hom_{\ZZ_p[G]}(A(F)_p,\ZZ_p[G])$. We fix a pre-image $\phi_i'$ of each $\phi_i$ under the surjective homomorphism 
%
%
occurring in the canonical short exact sequence
\begin{equation}\label{sha-selmer}
 \xymatrix{0 \ar[r] & \sha(A_F)[p^\infty]^\vee \ar[r] & \Sel_p(A_F)^\vee \ar[r]&  \Hom_{\ZZ_p}(A(F)_p,\ZZ_p)\ar[r] & 0.}
\end{equation}
We set $\phi'_\bullet:=(\phi'_i)_{1\leq i\leq a}$ and view this $a$-tuple as comprising elements of $Y_\pi$. 

Then the non-abelian higher special element associated to the data $(C_{S,X,T},(h_{A,F}^T)^{-1},\mathcal{L}_T,\pi,\phi'_\bullet)$ coincides with the pre-image under the bijective map $\bigwedge_{\CC_p[G]}^a\iota_{\CC_p}$ of the element
$$(\prod_{v \in T}\iota_\#(P_v(A_{F/k},1))) 
\cdot\frac{L^{(a)}_{S}(A_{F/k},1)}{\Omega_A^{F/k}\cdot w_{F/k}^d}\cdot \mathcal{LR}^p_{A^t_{F/k}}(x_\bullet)\cdot {\rm ht}^{(a)}_{A_{F/k}}(\wedge_{i=1}^{i=a}\phi_i).$$


We fix an $a$-tuple $\theta_\bullet$ in $\Hom_{\ZZ_p[G]}(A^t(F)_p,\ZZ_p[G])$, and identify it with its image under the injective map
\[ \Hom_{\ZZ_p[G]}(A^t(F)_p,\ZZ_p[G]) \to \Hom_{\ZZ_p[G]}(H^1(C_{S,X,T}),\ZZ_p[G])\]
induced by $\iota$.

%
%
Then Corollary \ref{cor2} implies that any element of the form (\ref{key product}) belongs to $${\rm Ann}_{\ZZ_p[G]}((Y_\pi)_{\rm tor})\subseteq{\rm Ann}_{\ZZ_p[G]}((\Sel_p(A_F)^\vee)_{\rm tor}).$$

To complete the proof of Theorem \ref{big conj} it is therefore enough to note that the exact sequence (\ref{sha-selmer}) identifies
 $(\Sel_p(A_F)^\vee)_{\rm tor}$ with $\sha(A_F)[p^\infty]^\vee$ and that the Cassels-Tate pairing identifies $\sha(A_F)[p^\infty]^\vee$ with $\sha(A^t_F)[p^\infty]$.

\subsection{Dihedral twists of elliptic curves}\label{dihedralsect}

In this section we assume that $A$ is an elliptic curve and that $F/k$ is a generalised dihedral extension in the sense of Mazur and Rubin \cite{mr2}.

We recall that this condition means that $G$ has an abelian normal Sylow $p$-subgroup $P$ of index two (where $p$ still denotes an odd prime) and that the conjugation action of any lift to $G$ of the generator of $G/P$ inverts elements of $P$. In the sequel we fix such a lift $\tau\in G$ and also set $K:=F^P$. We write $\epsilon$ for the unique linear non-trivial character of $G$. 

\subsubsection{Statement of the main result}
For any $G$-module $M$ we set $M':=\ZZ[\frac{1}{2}]\otimes_{\ZZ} M$. For any $\psi\in{\rm Ir}(G)$ we write $T_\psi$ for the element $\sum_{g\in G}\psi(g^{-1})\cdot g$ of $Z(\CC_p[G])$ and also use the product of periods $\Omega_A^\psi$ and Artin root number $w_\psi$ specified in \cite[\S 4.1]{rbsd}. We recall that, in particular, one has
$$\Omega^{{\bf 1}_G}_A=\Omega_A^+:=\prod_{v\in S_\CC(k)}\Omega_{A,v}\cdot\prod_{v\in S_\RR(k)}\Omega_{A,v}^+$$ and
$$\Omega^\epsilon_A=\Omega_A^-:=\prod_{v\in S_\CC(k)}\Omega_{A,v}\cdot\prod_{v\in S_\RR(k)}\Omega_{A,v}^-$$
with each individual term $\Omega_{A,v},\Omega_{A,v}^+,\Omega_{A,v}^-$ defined in loc. cit. as a fully explicit classical period.

We still assume to be given a set $S$ of places of $k$ as in (\ref{starSset}).
We again abbreviate $S_{\rm ram}(F/k)$ to $S_{\rm ram}$, while $S_{\rm ram}^{\rm sp}$ will denote the subset of $S_{\rm ram}(F/k)$ comprising places which split in $K/k$. In addition $d_k$ and $d_K$ denote the discriminants of $k$ and of $K$ respectively while $Nf(\phi)$ is be the absolute norm of the Artin conductor of any $\phi$ in ${\rm Ir}(P)$. In addition we will use the `unramified characteristic' $$u_\psi:=\prod_{v\in S_{\rm ram}}{\rm det}(-\Phi_v^{-1}\mid V_\psi^{I_v})$$ for each $\psi\in{\rm Ir}(G)$.

If $Q$ is a given point in $A(F)$ which satisfies $\tau(Q)=(-1)^{i_Q}Q$ with $i_Q\in\{0,1\}$ then for any $\psi\in{\rm Ir}(G)$ we set
$$h_{F,\psi}(Q):=\begin{cases}

\medskip

1, \,\,\,\,\,\,\,\,\,\,\,\,\,\,\,\,\,\,\,\,\,\,\,\,\,\,\,\,\,\,\,\,\,\,\,\,\,\,\,\,\text{ if either }\psi={\bf 1}_G\text{ and }i_Q=1,\text{ or }\psi=\epsilon\text{ and }i_Q=0,\\
\psi(1)|G|^{-1}\langle T_\psi(Q),T_{\check\psi}(Q)\rangle_{A_F},\,\,\,\,\,\,\,\,\,\,\,\,\,\,\,\,\, \,\,\,\,\,\,\,\,\,\,\,\,\,\,\,\,\,\,\,\,\,\,\,\,\,\,\,\,\,\,\,\,\,\,\,\,\,\,\,\,\,\,\,\,\,\,\,\,\,\,\,\,\,\,\,\,\text{ otherwise. }
\end{cases}$$
and then also
$$\mathcal{Q}_\psi:=\begin{cases}

\bigskip

(-1)^{|S_{\rm ram}|}\cdot\sqrt{|d_k|}\cdot\frac{L'_{S_{\rm ram}}(A,1)}{\Omega_A^+\cdot h_{F,{\bf 1}_G}(Q)}, \,\,\,\,\,\,\,\,\,\,\,\,\,\,\,\,\,\,\,\,\,\,\,\,\,\text{ if }\psi={\bf 1}_G,\\ 

\bigskip

(-1)^{|S^{\rm sp}_{\rm ram}|}\cdot\sqrt{|d_K/d_k|}\cdot\frac{L'_{S_{\rm ram}}(A,\epsilon,1)}{\Omega_A^-\cdot h_{F,\epsilon}(Q)}, \,\,\,\,\,\,\,\,\,\,\,\,\,\,\,\text{ if }\psi=\epsilon,\\
u_\psi\cdot\sqrt{|d_K|\cdot Nf(\phi)}\cdot\frac{L'_{S_{\rm ram}}(A,\check\psi,1)}{\Omega_A^\psi\cdot h_{F,\psi}(Q)}, \,\,\,\,\,\,\,\,\,\,\,\,\,\,\,\,\,\,\,\,\text{ if }\psi={\rm Ind}_P^G(\phi)\text{ with }\phi\in{\rm Ir}(P).
\end{cases}.$$

The proof of this result will be given in \S \ref{dihedralproof} below.

\begin{theorem}\label{dihedralthm} Let $F/k$ be a generalised dihedral extension, with $K:=F^P$ as above, and let $A$ be an elliptic curve. Assume that $F/k$ is unramified at all places of $k$ at which $A$ has bad reduction, that all $p$-adic places of $k$ split completely in the quadratic extension $K/k$ and that the rank of $A(K)$ is odd.
\begin{itemize} \item[(i)] Then there exists a point $Q$ in $A(F)$ which satisfies $\tau(Q)=(-1)^{i_Q}Q$ with $i_Q\in\{0,1\}$ and generates a $\ZZ'[G]$-submodule $\langle Q\rangle$ of $A(F)'$ isomorphic to $\ZZ'[G](1+(-1)^{i_Q}\tau)$.

If in addition one has $A(K)[p]=0$ and also $$(1+(-1)^i\tau)\Tr_{F/K}(A(F))\nsubseteq p\cdot A(K)$$ for some $i\in\{0,1\}$, then one may choose the point $Q$ so that $i_Q=i$ and $\langle Q\rangle_p$ is a direct summand of the $\ZZ_p[G]$-module $A(F)_p=A(F)_{p,{\rm tf}}$.
\end{itemize}

In the sequel fix a point $Q$ as in claim (i) and write $t_Q$ for the exponent of the group $\bigl(A(F)_{p,{\rm tf}}/\langle Q\rangle_p\bigr)_{\rm tor}$.



\begin{itemize}
\item[(ii)] There is an element $\phi_Q$ of $\Hom_{\ZZ_p[G]}(A(F)_p,\ZZ_p[G])$ with the property that
$$e_{(1)}\cdot(\wedge_{j=1}^{j=1}\phi_Q)({\rm ht}_{A_{F/k}}^{(1)}(\wedge_{i=1}^{i=1}\phi_Q))=e_{(1)}\cdot\sum_{\psi\in{\rm Ir}(G)}t_Q^{2\psi(1)}\cdot h_{F,\psi}(Q)^{-1}\cdot e_\psi.$$

\end{itemize}

In the sequel we assume that $A(K)[p]=0$ and that the hypotheses (H$_1)$-(H$_4$) that are listed in \S\ref{bkexample} are satisfied.

\begin{itemize}
\item[(iii)] 
%
%
%
%

The $R_{(1)}$-module $R_{(1)}\otimes_{\ZZ_p[G]}\Sel_p(A_F)^\vee$ admits a free, quadratic presentation $\Pi_{\rm Sel}$ with the property that,
if $p$ is unramified in $F/\QQ$ and the conjecture ${\rm BSD}(A_{F/k})$ of \cite{rbsd} is valid, then the product
\begin{equation}\label{explicitkeyproduct}{\rm nr}_{\QQ_p[G]e_{(1)}}(y)^2
\cdot\bigl(\sum_{\psi\in{\rm Ir}(G)}t_Q^{2\psi(1)}\cdot\mathcal{Q}_\psi\cdot e_\psi\bigr)\end{equation}
belongs to ${\rm Fit}^1_{R_{(1)}}(\Pi_{\rm Sel})$ for any $y\in I_{(1)}$ and, after multiplication by any element of $\alpha\in I_{(1)}\cdot\mathcal{D}(R_{(1)})$, also to ${\rm Ann}_{\ZZ_p[G]}(\sha(A_F)[p^\infty])$.

In particular, we fix any elements $\delta_\psi$ in the inverse differents of the fields generated over $\QQ_p$ by the values of each $\psi\in{\rm Ir}(G)$. We set $m_\psi:=n(2\psi(1)+1)$.

Then if ${\rm rk}(A(k))=0$, one has
$$\sum_{\psi\neq{\bf 1}_G}p^{m_\psi}t_Q^{2\psi(1)}\check\psi(g)\delta_\psi\mathcal{Q}_\psi\in\ZZ_p$$ for every $g\in G$.

If instead ${\rm rk}(A(k))={\rm rk}(A(K))=1$, one has
$$\sum_{\psi\neq\epsilon}p^{m_\psi}t_Q^{2\psi(1)}\check\psi(g)\delta_\psi\mathcal{Q}_\psi\in\ZZ_p$$
for every $g\in G$.

In all other cases, one has
$$\sum_{\psi\in{\rm Ir}(G)}t_Q^{2\psi(1)}\check\psi(g)\delta_\psi\mathcal{Q}_\psi\in\ZZ_p$$
for every $g\in G$.

\item[(iv)] If $P$ is cyclic and for every subgroup $H$ of $P$ the restriction map $\sha(A_{F^H})[p^\infty]\to\sha(A_F)[p^\infty]$ is injective, then one may choose the point $Q$ so that $\langle Q\rangle_p$ is a direct summand of the $\ZZ_p[G]$-module $A(F)_p=A(F)_{p,{\rm tf}}$, so in particular $t_Q=1$.
\end{itemize}
\end{theorem}
\begin{remarks}{\em \

\noindent{}(i) The proof of the final assertions of claim (iii) will rely on making explicit choices of $y\in I_{(1)}$ and of $\alpha\in I_{(1)}\cdot\mathcal{D}(R_{(1)})$ which would also determine fully explicit elements of ${\rm Fit}^1_{R_{(1)}}(\Pi_{\rm Sel})$ and of ${\rm Ann}_{\ZZ_p[G]}(\sha(A_F)[p^\infty])$.

\noindent{}(ii) Both claim (ii) of Theorem \ref{dihedralthm}, when combined with Theorem \ref{big conj}, and claim (iii) of Theorem \ref{dihedralthm}, extend the congruence relations predicted by Burns, Wuthrich and the first author in Theorem 5.8 of \cite{bmw}. We recall that in loc. cit., in addition to all the hypotheses of claim (iii), it was assumed that the rank of $A(K)$ is equal to 1 and that $\sha(A_K)[p^\infty]$ vanishes.
In particular, the assertions of claim (iii) hold unconditionally for the families for which the conjecture ${\rm BSD}(A_{F/k})$ was verified, either theoretically or numerically, via Theorem 5.8 of loc. cit..

\noindent{}(iii) Moreover, our methods lead to more general, albeit less explicit, versions of claim (iii) which do not require $p$ to be unramified in $F/\QQ$ while at the same time circumventing the need for the logarithmic resolvents occurring in the much more general claim (ii).

For instance, if $p$ is unramified in $k/\QQ$ rather than in $F/\QQ$ (so in particular if $k=\QQ$), then under the assumed compatibility of specific cases of the `local epsilon constant conjecture', and the assumed validity of a conjecture of Breuning \cite[Conj. 3.2]{20}, the given claims for the element (\ref{explicitkeyproduct}) would remain true after multiplying it by the product, over all $p$-adic places of $k$ that are wildly ramified in $F$, of the fully explicit elements occurring in the display \cite[(51)]{rbsd}.

In this regard we recall that work of Breuning \cite{19,20}, of Bley and Debeerst \cite{15} and of Bley and Cobbe \cite{13} provides verifications of Breuning's conjecture for natural families of dihedral extensions of $\QQ$ in which $p$ is wildly ramified. Similarly, Bley and Cobbe \cite{14} have proved the compatibility of the relevant cases of the local epsilon constant conjecture for certain families of wildly ramified extensions. See \cite[Rem. 6.7, Rem. 6.8, \S 5.3]{rbsd} for a more detailed account.


}\end{remarks}

\begin{example}\label{heegner}{\em Assume that the elliptic curve $A$ is defined over $k=\QQ$ and has conductor $N$. We fix a modular parametrisation $\varphi_A:X_0(N)\to A$ of minimal degree.

We assume that the given generalised dihedral extension $F$ of $\QQ$ contains an imaginary quadratic field $K$ in which all prime divisors of $N$ split, and in addition that the conductor $c$ of $F$ is a square-free product of primes that are both inert in $K$ and coprime to $N$.

We write $K_c$ for the ring class field of $K$ of conductor $c$ and set $G_c:={\rm Gal}(K_c/K)$. Then a standard construction described in \cite[\S 12.1]{rbsd} gives a canonical point $x_c$ on $X_0(N)(K_c)$ and we follow loc. cit. in defining a `higher Heegner point' $y_c:=\varphi_A(x_c)$ of $A(K_c)$.

Assuming $L(A_K,z)$ vanishes to order one at $z=1$, Zhang's generalisation \cite{zhang01, zhang} of the seminal results of Gross and Zagier in \cite{GZ} implies that for every $\phi$ in ${\rm Ir}(G_c)$ the function $L(A_K,\phi,z)$ vanishes to order one at $z=1$ and that the number $\langle T_{\phi}(y_c),T_{\check\phi}(y_c)\rangle _{A_{K_c}}$ is non-zero.

In this setting one therefore knows that the rank of $A(K)$ is equal to 1 and, via the proof of this result given in \S \ref{dihedralproof} below, also that the point $Q:=\Tr_{K_c/F}(y_c)$ of $A(F)$ satisfies the conditions of claim (i) of Theorem \ref{dihedralthm}.

If we now assume that $A(K)[p]=0$ (with the prime $p$ determined by the generalised dihedral extension $F/\QQ$), that the hypotheses (H$_1)$-(H$_4$) that are listed in \S\ref{bkexample} are satisfied and that the point $\Tr_{K_1/K}(y_1)$ of $A(K)$ is not divisible by $p$, then it is easy to see that the number $t_Q$ occurring in the expression (\ref{explicitkeyproduct}) is equal to 1.

It follows that claim (iii) of Theorem \ref{dihedralthm} describes explicit relationships, between the arithmetic of higher Heegner points of $A$ in generalised dihedral extensions $F/\QQ$ and the Galois module structure of Tate-Shafarevich and Selmer groups in $F/\QQ$, that are encoded in the refined Birch and Swinnerton-Dyer conjecture ${\rm BSD}(A_{F/\QQ})$ of \cite{rbsd}. In particular, this claim extends the description of such relationships that is given in Theorem 12.2 (ii) of loc. cit., which is instead solely concerned with the ${\rm Gal}(F/K)$-module structure of $\sha(A_F)$.

In fact in the setting of this Example, it is straightforward to prove a converse to Theorem \ref{dihedralthm}(iii) that gives an explicit criterion to verify the validity of (the $p$-component) of conjecture ${\rm BSD}(A_{F/\QQ})$ in terms of the properties of the point $Q=\Tr_{K_c/F}(y_c)$. This approach will be developed in greater generality in future work.
}\end{example}

\subsubsection{The proof of Theorem \ref{dihedralthm}}\label{dihedralproof} To prove claim (i) we use the result \cite[Thm. B]{mr2} of Mazur and Rubin, which implies that the $\QQ_p[P]$-module $\QQ_p\otimes_{\ZZ_p}\Sel_p(A_F)^\vee\cong\QQ_p\otimes_{\ZZ_p}A(F)_p^*$ has a direct summand that is isomorphic to $\QQ_p[P]$. The same is thus true of the $\QQ_p[P]$-module $\QQ_p\otimes_{\ZZ_p}A(F)_p$.

The $\QQ_p[G]$-submodule generated by this summand inside of $\QQ_p\otimes_{\ZZ_p}A(F)_p$ must then contain a copy of either $\QQ_p[G](1+\tau)$ or of $\QQ_p[G](1-\tau)$. Setting $e_{\pm}:=\frac{1\pm\tau}{2}$ we then deduce the existence of an injective homomorphism of $\ZZ_p[G]$-modules $\ZZ_p[G]e_\pm\to A(F)_p$. The existence of the claimed point $Q$ then follows easily from Roiter's Lemma \cite[(31.6)]{curtisr}.

Now, if one has $(1+(-1)^i\tau)\Tr_{P}(A(F))\nsubseteq p\cdot A(K)$ for some $i\in\{0,1\}$, proving that one may choose the point $Q$ with $i_Q=i$ is a relatively straightforward exercise which we leave to the reader. Moreover, after replacing such a point $Q$ by $(1+(-1)^{i_Q})Y+p^N Q$ if necessary, for any given point $Y$ of $A(F)$ with $(1+(-1)^{i_Q})\Tr_{P}(Y)\notin p\cdot A(K)$ and a large enough positive integer $N$, one may assume in addition that $\Tr_{P}(Q)\notin p\cdot A(K)$.

To conclude the proof of claim (i) we shall now use this condition, together with the assumption $A(K)[p]=0$, to verify that $\langle Q\rangle_p$ is a direct summand of the $\ZZ_p[G]$-module $A(F)_p$. We first observe that the two given conditions imply that the quotient $A(K)_p/\ZZ_p\cdot\Tr_{P}(Q)$ is $\ZZ_p$-free.

But $\langle Q\rangle_p$ is a cohomologically-trivial $P$-module and so one has a canonical isomorphism
$$A(K)_p/\langle Q\rangle_p^P\cong(A(F)_p/\langle Q\rangle_p)^P.$$
Since in addition $\langle Q\rangle_p^P=\ZZ_p\cdot\Tr_{P}(Q)$ we deduce that $(A(F)_p/\langle Q\rangle_p)^P$ is also $\ZZ_p$-free. Since $P$ is a $p$-group, the quotient $A(F)_p/\langle Q\rangle_p$ must itself be $\ZZ_p$-free.

This last fact implies that there is a canonical isomorphism
$${\rm Ext}_{\ZZ_p[G]}^1(A(F)_p/\langle Q\rangle_p,\langle Q\rangle_p)\cong H^1(G,\Hom_{\ZZ_p}(A(F)_p/\langle Q\rangle_p,\langle Q\rangle_p))$$ and, since the latter group vanishes (because $\langle Q\rangle_p$ is a projective $\ZZ_p[G]$-module), we finally find that $\langle Q\rangle_p$ is indeed a direct summand of the $\ZZ_p[G]$-module $A(F)_p$. This completes the proof of claim (i).

We now consider claim (ii). We 
observe first that there is a canonical exact sequence
$$\Hom_{\ZZ_p[G]}(A(F)_p,\ZZ_p[G])\to\Hom_{\ZZ_p[G]}(\langle Q\rangle_p,\ZZ_p[G])\to\bigl(A(F)_{p,{\rm tf}}/\langle Q\rangle_p\bigr)_{\rm tor}^\vee.$$
We define $Q^*\in\Hom_{\ZZ_p[G]}(\langle Q\rangle_p,\ZZ_p[G])$ by setting $Q^*(Q):=1$ and also $Q^*(\pi\cdot Q):=0$ for each $\pi\in P$. We then let $\phi_Q$ be any element of $\Hom_{\ZZ_p[G]}(A(F)_p,\ZZ_p[G])$ whose restriction to $\langle Q\rangle_p$ is equal to $t_Q\cdot Q^*$.


Then an explicit computation shows that
\begin{align}\label{tqcomputation}&e_{(1)}\cdot(\wedge_{j=1}^{j=1}\phi_Q)({\rm ht}_{A_{F/k}}^{(1)}(\wedge_{i=1}^{i=1}\phi_Q))\notag\\ =&e_{(1)}\cdot{\rm nr}_{\QQ_p[G]}(t_Q)^2\cdot{\rm nr}_{\CC_p[G]e_{(1)}}\bigl(e_{(1)}\cdot(\sum_{g\in G}\langle g(Q),Q\rangle_{A_F}\cdot g^{-1})\bigr)^{-1}\notag\\
=&e_{(1)}\cdot\bigl(\sum_{\psi\in{\rm Ir}(G)}t_Q^{\psi(1)}\cdot e_\psi\bigr)^2\cdot\bigl(\sum_{\psi\in{\rm Ir}(G)}h_{F,\psi}(Q)\cdot e_\psi\bigr)^{-1},\end{align}
as required.


We now assume that $A(K)[p]=0$ and that the hypotheses (H$_1)$-(H$_4$) that are listed in \S\ref{bkexample} are satisfied, and proceed to deduce the validity of claim (iii) from the variant of Theorem \ref{big conj} that is given in Remark \ref{new remark SC}.

In fact, in view of the computation (\ref{tqcomputation}) and of the explicit definitions of $L^{(1)}_{S}(A_{F/k},1)$, of $\Omega_A^{F/k}:=\sum_{\psi\in{\rm Ir}(G)}\Omega_A^\psi\cdot e_\psi$ and of $w_{F/k}:=\sum_{\psi\in{\rm Ir}(G)}w_\psi\cdot e_\psi$, the first assertion of claim (iii) is valid because if $p$ is unramified in $F/\QQ$ then
$$\frac{L^{(1)}_{S_{\rm ram}}(A_{F/k},1)}{\Omega_A^{F/k}\cdot w_{F/k}}\cdot\tau^*(F/k)\cdot\bigl(\sum_{\psi\in{\rm Ir}(G)}h_{F,\psi}(Q)\cdot e_\psi\bigr)^{-1}=\sum_{\psi\in{\rm Ir}(G)}\mathcal{Q}_\psi\cdot e_\psi.$$ This equality is itself an immediate consequence of the computation of the term $\tau^*(F/k)\cdot w_{F/k}^{-1}$ that is carried out in \cite[(21)]{bmw}.

As for the final assertion of claim (iii), we use the fact that, for any family $(C_\psi)_{\psi\in{\rm Ir}(G)}$ of elements of $\CC_p$, the sum \begin{equation}\label{observationbelow}\sum_{\psi\in{\rm Ir}(G)} C_\psi\cdot e_\psi=\sum_{g\in G}\bigl(|G|^{-1}\sum_{\psi\in{\rm Ir}(G)}\psi(1)\check\psi(g)C_\psi\bigr)\cdot g\end{equation} belongs to $\ZZ_p[G]$ if and only if, for each $g\in G$, the sum $\sum_{\psi\in{\rm Ir}(G)}\psi(1)\check\psi(g)C_\psi$ belongs to $|G|\cdot\ZZ_p$.

We will combine this observation with the first assertion of claim (iii). We first observe that from claim (i) one easily deduces that
\begin{equation*}\label{e(1)}e_{(1)}=\begin{cases}1-e_{{\bf 1}_G},\,\,\,\,\,\,\,\,\,\,\,\,\,\,\,\,\,\,\,\,\,\,\,\,\,\,\,\,\,\,\,\,\text{if }{\rm rk}(A(k))=0,\\
1-e_{\epsilon},\,\,\,\,\,\,\,\,\,\,\,\,\,\,\,\,\,\,\,\,\,\,\,\,\,\,\,\,\,\,\,\,\,\,\,\,\text{if }{\rm rk}(A(k))={\rm rk}(A(K))=1,\\
1,\,\,\,\,\,\,\,\,\,\,\,\,\,\,\,\,\,\,\,\,\,\,\,\,\,\,\,\,\,\,\,\,\,\,\,\,\,\,\,\,\,\,\,\,\,\,\,\,\text{otherwise}.\end{cases}\end{equation*}
We may thus use the element
\begin{equation*}\label{ychoice}y=\begin{cases}p^n(1-e_{{\bf 1}_G}),\,\,\,\,\,\,\,\,\,\,\,\,\,\,\,\,\,\,\,\,\,\,\,\,\,\,\,\,\,\,\text{if }{\rm rk}(A(k))=0,\\
p^n(1-e_{\epsilon}),\,\,\,\,\,\,\,\,\,\,\,\,\,\,\,\,\,\,\,\,\,\,\,\,\,\,\,\,\,\,\,\,\,\,\text{if }{\rm rk}(A(k))={\rm rk}(A(K))=1,\\
1,\,\,\,\,\,\,\,\,\,\,\,\,\,\,\,\,\,\,\,\,\,\,\,\,\,\,\,\,\,\,\,\,\,\,\,\,\,\,\,\,\,\,\,\,\,\,\,\,\,\,\,\,\,\,\,\,\,\text{otherwise}\end{cases}\end{equation*}
of $I_{(1)}$ to construct the product (\ref{explicitkeyproduct}).
In addition, for any given elements $\delta_\psi$ as in the statement of this final assertion, the arguments of Johnston and Nickel in \cite[\S 6.4]{jn} then imply that we may use the element $$\alpha:=y\cdot|G|\sum_{\psi\in{\rm Ir}(G)}\psi(1)^{-1}\delta_\psi e_\psi$$ of $I_{(1)}\cdot\mathcal{D}(R_{(1)})$.

From the first assertion of claim (iii) we then know that the element
\begin{align*}&\alpha\cdot{\rm nr}_{\QQ_p[G]e_{(1)}}(y)^2\cdot\bigl(\sum_{\psi\in{\rm Ir}(G)}t_Q^{2\psi(1)}\cdot\mathcal{Q}_\psi\cdot e_\psi\bigr)\\ 
=&y\cdot{\rm nr}_{\QQ_p[G]e_{(1)}}(y)^2\cdot\bigl(\sum_{\psi\in{\rm Ir}(G)}|G|\cdot t_Q^{2\psi(1)}\cdot\psi(1)^{-1}\cdot\delta_\psi\cdot\mathcal{Q}_\psi\cdot e_\psi\bigr)\end{align*}
belongs to ${\rm Ann}_{\ZZ_p[G]}(\sha(A_F)[p^\infty])\subseteq\ZZ_p[G]$.

Now for our chosen element $y$ we have
$$y\cdot{\rm nr}_{\QQ_p[G]e_{(1)}}(y)^2=\begin{cases}\sum_{\psi\neq{\bf 1}_G}p^{n(2\psi(1)+1)}e_\psi,\,\,\,\,\,\,\,\,\,\,\,\,\,\,\,\,\,\,\,\,\,\,\,\,\,\,\,\text{if }{\rm rk}(A(k))=0,\\
\sum_{\psi\neq\epsilon}p^{n(2\psi(1)+1)}e_\psi,\,\,\,\,\,\,\,\,\,\,\,\,\,\,\,\,\,\,\,\,\,\,\,\,\,\,\,\,\,\,\,\text{if }{\rm rk}(A(k))={\rm rk}(A(K))=1,\\
1,\,\,\,\,\,\,\,\,\,\,\,\,\,\,\,\,\,\,\,\,\,\,\,\,\,\,\,\,\,\,\,\,\,\,\,\,\,\,\,\,\,\,\,\,\,\,\,\,\,\,\,\,\,\,\,\,\,\,\,\,\,\,\,\,\,\,\,\,\,\,\,\,\,\,\,\,\,\,\,\text{otherwise},\end{cases}$$
and so the claimed explicit integrality conditions follow from the general argument given just below (\ref{observationbelow}).

We finally note that the validity of claim (iv) follows directly upon combining the result \cite[Thm. B]{mr2} of Mazur and Rubin used above with a (dual) variant of \cite[Thm. 2.7 (ii)]{bmcw} that replaces the dependence of this result on Proposition 3.1 from loc. cit. (which is formulated in terms of cokernels of norm maps on $p$-primary Tate-Shafarevich groups) by the use of the result \cite[Lem. 3.3]{ksdmwg} (formulated instead in terms of kernels of restriction maps).






\Addresses
\end{document}